\documentclass[12pt]{amsart}
\usepackage{amsmath}
\usepackage{amssymb}
\usepackage{amstext}
\usepackage{xcolor}
\usepackage{hyperref}
\usepackage{mathtools}
\usepackage{dsfont}
\usepackage{enumerate}
\usepackage{bbm}

\mathtoolsset{showonlyrefs}

\usepackage{tikz}

\newtheorem{theorem}{Theorem}[section]
\newtheorem{corollary}[theorem]{Corollary}
\newtheorem{lemma}[theorem]{Lemma}
\newtheorem{proposition}[theorem]{Proposition}
\newtheorem{definition}[theorem]{Definition}
\newtheorem{remark}[theorem]{Remark}

\newtheorem{assumption}[theorem]{Assumption}
\newtheorem{convention}[theorem]{Convention}

\DeclareMathOperator{\Ent}{Ent}
\DeclareMathOperator{\Id}{Id}
\DeclareMathOperator{\Hess}{Hess}

\DeclareMathOperator{\id}{id}
\DeclareMathOperator{\cov}{cov}
\DeclareMathOperator{\Image}{Im}

\newcommand{\R}{\mathbb{R}}
\newcommand{\T}{\mathbb{T}}

\begin{document}

 \title[The quantitative hydrodynamic limit]{The quantitative hydrodynamic limit of the Kawasaki dynamics}

\author{Deniz Dizdar}
\address{Universit\'e de Montr\'eal}
\email{deniz.dizdar@umontreal.ca}

\author{Georg Menz}
\address{University of California, Los Angeles}
\email{gmenz@math.ucla.edu}

\author{Felix Otto}
\address{Max Planck Institute for Mathematics in the Sciences, Leipzig, Germany}
\email{Felix.Otto@mis.mpg.de.}

\author{Tianqi Wu}
\address{University of California, Los Angeles}
\email{timwu@ucla.edu}

\date{\today}

 \maketitle

 \begin{abstract}
  We derive for the first time in the literature a rate of convergence in the hydrodynamic limit of the Kawasaki dynamics for a one-dimensional lattice system. We use an adaptation of the two-scale approach. The main difference to the original two-scale approach is that the observables on the mesoscopic level are described by a projection onto splines of second order, and not by a projection onto piecewise constant functions. This allows us to use a more natural definition of the mesoscopic dynamics, which yields a better rate of convergence than the original two-scale approach. 
 \end{abstract}

 \begin{footnotesize}
  \end{footnotesize}

\tableofcontents

\section{Introduction} \label{intro}
The broader scope of this work is the derivation of scaling limits for lattice systems. Typically, such a result consists in showing that under a suitable time-space rescaling a random evolution of a lattice system converges to a macroscopic evolution as the system size goes to infinity. One considers two different cases of limits. The hydrodynamic limit is a dynamical version of the law of large numbers. The limiting macroscopic evolution is deterministic and describes the typical macroscopic behavior of the system. In the fluctuation limit, the limiting macroscopic evolution is random and it describes the fluctuations around the hydrodynamic limit. \\

In this work, we are interested in the hydrodynamic limit of the Kawasaki dynamics of one-dimensional lattice systems of continuous, unbounded spins.  The Kawasaki dynamics is a spin-exchange dynamic preserving the mean spin. In the hydrodynamic limit one shows that it converges to a non-linear heat equation. The hydrodynamic limit is long known on a qualitative level. It was first deduced by Fritz~\cite{Fri87}. Then, Guo, Papanicolaou and Varadhan~\cite{GPV} deduced the hydrodynamic limit by introducing the martingale method. In~\cite{LY},  Yau introduced the entropy method, which is based on a sophisticated Gronwall-type estimate for a relative entropy functional. Yau's method is simpler and gives stronger results, but it makes stronger assumptions on the initial data (closeness to hydrodynamic behavior in the sense of relative entropy rather than in the sense of macroscopic observables). All those methods are qualitative and it's not obvious how to make them quantitative.\\

In this article, we further develop the quantitative theory on the hydrodynamic limit i.e.~establising rates of convergence in the hydrodynamic limit. A first step toward a quantitative theory was achieved in~\cite{GORV} by introducing the two-scale approach. For a detailed description of the two-scale approach we refer to Section~\ref{s_two_scale_approach}. In a nutshell, the two-scale approach introduces an additional mesoscopic scale in between the microscopic and the macroscopic scale. The hydrodynamic limit is then deduced in two steps. First showing the closeness of the microscopic and a carefully chosen mesoscopic dynamics and then the closeness of the mesoscopic and the macroscopic dynamics. However, in~\cite{GORV} the hydrodynamic limit is still deduced on a qualitative level. The main estimate establishing the closeness of the microscopic and the mesoscopic dynamics is already quantitative. The second estimate, showing the closeness of the mesoscopic and the macroscopic dynamics, is in principle just numerical analysis. With some work, one could make this estimate quantitative as well, overall deducing a quantitative result on the hydrodynamic limit.\\

In this article, we establish for the first time in the literature quantitative error estimates for the hydrodynamic limit. Instead of completing the approach of~\cite{GORV} we proceed differently. The reason is that when using the approach of~\cite{GORV}, the resulting error estimates woud be sub-optimal (for details see Remark~\ref{r_remark_choosing_L_2} and Remark~\ref{optimalsc} below). The sub-optimality comes from the fact that~\cite{GORV} uses a projection onto piecewise constant functions to define the mesoscopic scale. By lack of regularity, this choice forces one to use an unnatural definition of the mesoscopic dynamics. Overall, one has to use a mixed Galerkin procedure. In this work, we use a projection onto splines to define the mesoscopic observables. By this choice we are able to define the mesoscopic dynamics in a natural way as the Galerkin approximation of the macroscopic dynamics. This leads to better error estimates compared to~\cite{GORV}. However, because splines do not have a localized basis, deducing the ingredients of the two scale approach becomes much more subtle. To keep this article short, the verification of some of those ingredients is outsourced to the companion article~\cite{DMOW18b}. There, we deduce the strict convexity of the coarse-grained Hamiltonian, a uniform logarithmic Sobolev inequality and the convergence of the gradient of the free energies. In this article, we concentrate on showing the quantitative error bounds for the hydrodynamic limit.\\

The second motivation behind improving the estimates of~\cite{GORV} is to develop a quantitative theory of the fluctuation limit, which states that the fluctuations of the Kawasaki dynamics converge to the solution of a stochastic heat equation. As for the hydrodynamic limit, the fluctuation limit of the Kawasaki dynamic is well understood on a qualitative level (see for example~\cite{Spo86,Zhu90,ChYa92,DiGuPe17}), but there is no quantitative result. A possible line of attack would be to use the two-scale approach. The estimates of~\cite{GORV} for estimating the distance of the microscopic and mesoscopic dynamics are too weak when using the scaling of the fluctuation limit. Because our error terms scale better, our estimates are still meaningful under this scaling (cf.~Theorem~\ref{p_quantitative_micro_to_meso}). The authors will further investigate this direction in another work.\\

Another question that is asked in this setting is the convergence of the microscopic entropy to the hydrodynamic entropy, which is again well understood from a qualitative point of view (cf.~\cite{Kos01,Fa13}). With the provided tools, one could make the approach of Fathi~\cite{Fa13} quantitative. However, this direction still needs further investigation.

\section*{Notations and conventions}

\begin{itemize}
\item We use the letter~$C$ to denote a universal generic constant~$0< C< \infty$ that is independent of the dimension~$N$ of the underlying lattice.
\item We denote with~$a \lesssim b$ that $a \leq C b$.
\item We denote with~$a \cdot b$ and~$|\cdot|$ the standard Euclidean inner product and norm on~$\mathbb{R}^N$.
\item Let~$X$ be a Euclidean space and~$f: X \to \mathbb{R}$. Then we denote with~$\nabla f$ and~$\Hess f$ the gradient and Hessian inherited from the Euclidean structure of~$X$.
\item We use $dx$ as a shorthand for the Hausdorff or
Lebesgue measure of appropriate dimension.
\item $|\cdot|_{H^{1}}$ denotes the homogeneous~$H^1$ norm.
\item $\Phi(z) : = z \log z$.
\item $[M]:= \left\{1, \ldots, M \right\}$.
\item $L^2(\mathbb{T})$ denotes the $L^2$ functions on the torus $\mathbb{T} = [0, 1]$ with mean $0$.  
\end{itemize}

\section{Setting and main result: The hydrodynamic limit of the Kawasaki dynamics}\label{s_main_results}

We start with describing the Kawasaki dynamics on the microscopic lattice~$\left\{1, \ldots, N \right\}$. For this purpose, let us introduce the Hamiltonian~$H : \mathbb{R}^N \to \mathbb{R}$ of the system. It is given by
\begin{equation}
H_N(x) \,=\, \sum_{n=1}^N \psi(x_n). \label{hamiltonian}
\end{equation}
We assume that the function~$\psi: \mathbb{R} \to \mathbb{R}$ can be
written as 
\begin{equation}\label{e_structure_of_ss_potential}
 \psi(x_n)\,=\,  \frac{1}{2}\,x_n ^2 + a x_n + \delta\psi(x_n),
\end{equation}
where the function $\delta\psi \in C^2(\mathbb{R})$, $a$ is an arbitrary real number, and that there is a constant $C<\infty$ such that
\begin{equation}\label{e_assumptions_on_ss_perturbation}
\|\delta\psi\|_{L^{\infty}(\mathbb{R})} < C \:\: \text{and} \:\:\|\frac{d^2}{dx^2}\,
\delta\psi\|_{L^{\infty}(\mathbb{R})} < C.
\end{equation}
The function $\psi$ may be non-convex and it helps to think about the function $\psi$ as a double well-potential (see Figure~\ref{f_double_well}). \\

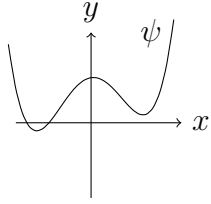
\begin{figure}[t]
  \centering
  \begin{tikzpicture}[scale=1]
      \draw[->] (-1,0) -- (1.2,0) node[right] {$x$};
      \draw[->] (0,-1) -- (0,1.2) node[above] {$y$};
      \draw[scale=0.5,domain=-2.2:2.2,smooth,variable=\x,black] plot
      ({\x},{0.3*(\x*\x-2)*(\x*\x-2)+0.15*\x});
      \draw (0.8,1.2) node {$\psi$};
    \end{tikzpicture}
  \caption{Double-well potential~$\psi$}\label{f_double_well}
\end{figure}

The Kawasaki dynamics~$X_t$ is given by the solution of the SDE
\begin{align}\label{e_Kawasaki_dynamics}
  dX_t = - A \nabla H (X_t) dt + \sqrt{2A} dB_t.
\end{align}
Here,~$B_t$ denotes a standard $N$-dimensional Brownian motion, ~$A$ denotes the second order difference operator of the periodic rescaled lattice~$\left\{ \frac{1}{N}, \ldots, 1  \right\}$. More precisely, the operator~$A$ is given by the $N\times N-$matrix
\begin{equation}\label{e_def_A}
A_{i,j} := N^2(-\delta_{i,j-1} + 2\delta_{i,j} - \delta_{i,j+1}),
\end{equation}
where we use the convention that~$0=N$. It follows from the structure of the operator~$A$ that the Kawasaki dynamics~~\eqref{e_Kawasaki_dynamics} conserves the mean spin of the system. Hence, we may restrict the state space~$\mathbb{R}^N$ of the Kawasaki dynamics~$X_t$ to the hyperplane
 \begin{equation}
    \label{e_definition_of_X_N_M}
  X_{N} := \left\{  x \in \mathbb{R}^N, \ \frac{1}{N} \sum_{i=1}^N x_i =m   \right\}.
 \end{equation}
We endow the space~$X_N$ with the Euclidean inner product
\begin{equation}
\langle x, y \rangle_{X_{N}} = x \cdot y =\underset{i = 1}{\stackrel{N}{\sum}} \hspace{1mm} x_i y_i.
\end{equation}

 \begin{assumption}\label{a_wlog_m_zero}
   By translating the single-site potential $\psi$ we may assume wlog.~that~$m=0$. Moreover, it follows from the structure of operator $A$ that the constant $a$ does not affect the dynamics, so that we may choose $a$ to additionally assume wlog.~that 
\begin{align}\label{e_wlog_mean_0}
\frac{\int_{\mathbb{R}}  z \exp (- \psi(z)) dz}{\int_{\mathbb{R}} \exp (- \psi(z)) dz } =0.\\
\end{align}
 \end{assumption}

The next lemma characterizes the law of the process~$X_t$ at time~$t$ via the Kolmogorov forward equation.

\begin{lemma}\label{p_fokker_planck_kawasaki}
Assume that the law of initial condition~$X_0$ is absolutely continuous wrt.~ the ~$N-1$ dimensional Hausdorff measure~$\mathcal{L}^{N-1}$. Let~$\mu$ denote the Gibbs measure on~$X_N$ associated to the Hamiltonian~$H$. More precisely, the measure~$\mu$ is absolutely continuous wrt.~the~$N-1$-dimensional Hausdorff measure~$\mathcal{L}^{N-1}$ and the Radon-Nikodym derivative of~$\mu$ is given by
\begin{align}\label{e_def_Gibbs_measure}
  \frac{d\mu}{d \mathcal{L}^{N-1}}(x) =  \frac{1}{Z} \exp \left( - H(x) \right) \qquad  x \in X_N.
\end{align}
Then for all times~$t>0$, the law~$p_t$ of the Kawasaki dynamics~$X_t$ given by~\eqref{e_Kawasaki_dynamics} is absolutely continuous wrt.~the Gibbs measure~$\mu$. Additionally, the relative density~$p(t) =f(t) \mu$ is a weak solution the Fokker Planck equation
\begin{equation} \label{micro_evolution}
\frac{\partial}{\partial t}(f\mu) = \nabla \cdot \left( A (\nabla f) \mu \right).
\end{equation}
This means that for any smooth test function $\xi$ it holds

$$\frac{d}{dt} \int{\xi(x)f(t,x)\mu(dx)} = -\int{\nabla \xi(x) \cdot A\nabla f(t,x) \mu(dx)}.$$
\end{lemma}
The statement of the last lemma follows from standard theory of stochastic processes (see for example \cite{Pav14}).\\

The goal of this article is to derive quantitative bounds on the hydrodynamic limit of the Kawasaki dynamics~$X_t \in X_N$. Hydrodynamic limit means that as~$N \to \infty$ the dynamics~$X_t$ defined on the discrete space~$X_N$ converges to a dynamics~$\zeta (t) $ on the one-dimensional torus~$\mathbb{T}=[0,1]$.
To do this, we embed the spaces $X_N$ into the space~$L^{2}(\mathbb{T})$ by identifying the vector~$x \in X_N$ with its corresponding step function on the interval~$[0,1]$.
\begin{convention}
	Given~$x \in X_N$, we identify it with the step function
	\begin{equation} \label{def_step_functions}
	x(\theta) = x_j, \hspace{1cm} \theta \in \left[ \frac{j-1}{N}; \frac{j}{N} \right).
	\end{equation}
  Then the space $X_N$ is identified with the space of piecewise constant functions on $\mathbb{T} = \R/\mathbb{Z}$ with mean~$0$, i.e.
\begin{align}\label{e_embedding_X}
X_{N} = &\left\{x : \mathbb{T} \longrightarrow \R; \hspace{2mm} x \text{ is constant on } \right. \\
  & \qquad  \left. \left[\frac{j-1}{N}; \frac{j}{N} \right), \hspace{2mm} j = 1,..,N,  \mbox{ and} \quad \int_0^1 x(\theta) d\theta = 0 \right\}.
\end{align}
\end{convention}

It turns out the $L^2$ norm is not well-suited to describe the hydrodynamic limit since it is too sensitive to local fluctuations. Therefore we endow the space $X_{N}$ with the weaker homogeneous~$H^{-1}$-norm.
\begin{definition}[$H^{-1}$-norm]\label{d_h_neg_norm}
  If $f : \mathbb{T} \rightarrow \R$ is a locally integrable function with mean 0 then
\begin{equation}
||f||_{H^{-1}}^2 := \int_{\T}{w(\theta)^2d\theta}, \hspace{5mm} w' = f, \hspace{3mm} \int_{\T}{w(\theta)d\theta} = 0.
\end{equation}
\end{definition}

We now describe the limiting macroscopic dynamics $\zeta(t)$. 

\begin{definition}[Macroscopic free energy] Let the function $\varphi: \mathbb{R} \to \mathbb{R}$ be defined by
	\begin{equation} \label{average_hydro_potential}
	\varphi(m) = \underset{\sigma \in \R}{\sup} \hspace{1mm} \left\{ \sigma m - \log \int_{\R}{\exp \left(\sigma x - \psi(x) \right)dx} \right\}.
	\end{equation}
	The macroscopic free energy~$\mathcal{H}: L^2 (\mathbb{T}) \to \mathbb{R}$ is given by
	\begin{align}
	\label{e_macroscopic_free_energy}
	\mathcal{H} (\zeta) = \int_{\mathbb{T}} \varphi (\zeta) d \theta.
	\end{align}
	It follows directly that
	$\nabla \mathcal{H} (\zeta) = \varphi'( \zeta)$ for any~$\zeta \in L^{2}(\mathbb{T})$.  
\end{definition}

\begin{definition} [Macroscopic dynamics]  The macroscopic dynamics $\zeta(t)$ is the unique weak solution of the equation
	\begin{equation} \label{hydro_equn}
	\frac{\partial \zeta}{\partial t} = \frac{\partial^2}{\partial \theta^2} \nabla \mathcal{H} (\zeta) = \frac{\partial^2}{\partial \theta^2}\varphi'(\zeta)
	\end{equation}
	with initial condition $\zeta(0,\cdot) = \zeta_0$. We defer the precise formulation to Definition~\ref{d_weak_solution} below.
\end{definition}

Now, let us formulate the main result of this article. 

\begin{theorem}[Quantitative hydrodynamic limit for the Kawasaki dynamics] \label{thm_lim_hydro}
We assume that the single-site potential $\psi$ satisfies~\eqref{e_structure_of_ss_potential} and~\eqref{e_assumptions_on_ss_perturbation}. Let~$\mu$ denote the Gibbs measure given by~\eqref{e_def_Gibbs_measure} and let~$f(t)\mu$ denote the law of the Kawasaki dynamics~$X_t$ (cf.~Lemma~\ref{p_fokker_planck_kawasaki}). We assume that the initial law $f(0) \mu$ of~$X_0$ has bounded microscopic entropy in the sense that for some constant $C_{\Ent} > 0$
\begin{equation}\label{e_closness_initial_data}
 \Ent(f(0)\mu | \mu) := \int {\left( \frac{df(0)\mu}{d\mu} \right) \log \left( \frac{df(0)\mu}{d\mu} \right) d \mu} \leq C_{\Ent} N.
\end{equation}
 
Let $\zeta(t)$ be the deterministic dynamics described by equation \eqref{hydro_equn}. Then there is a constant~$0<C< \infty$ depending only on the constants appearing in~\eqref{e_assumptions_on_ss_perturbation} such that for any $T > 0$ 
\begin{align}
  \label{e_quantitative_estimate_1}
  & \sup_{0 \leq t \leq T} \int |x - \zeta|_{H^{-1}}^2  f\mu(dx)   \\
&   \leq C   \int |x - \zeta (0)|_{H^{-1}}^2 f(0) \mu (dx) + \frac{C}{N^{\frac{2}{3}}} \Big[   T +  C_{\Ent} +  |\zeta (0)|_{L^2}^2 +  1 \Big].
\end{align}
\end{theorem}

The precise formulation of equation \eqref{hydro_equn} that describes the limiting macroscopic dynamics is given below.  

\begin{definition} \label{d_weak_solution}
	We call $\zeta(t,\theta)$ a weak solution of (\ref{hydro_equn}) on $[0,T] \times \T$ if
	\begin{equation}
	\zeta \in L^{\infty}_t(L^2_{\theta}), \hspace{5mm} \frac{\partial \zeta}{\partial t} \in L^2_t(H^{-1}_{\theta}), \hspace{5mm} \varphi'(\zeta) \in L^{\infty}_t(L^2_{\theta});
	\end{equation}
	and
	\begin{equation}\label{e_macro_weak_formulation}
	\left \langle \xi, \frac{\partial \zeta}{\partial t} \right\rangle_{H^{-1}} = - \left \langle \xi, \varphi'(\zeta) \right \rangle_{L^2} \hspace{5mm} \text{for all } \xi \in L^2, \hspace{3mm} \text{for a.e. } t \in [0,T].
	\end{equation}
	
	Here, $L^{\infty}_t(L^2_{\theta})$ (resp. $ L^2_t(H^{-1}_{\theta})$ is the set of functions $\zeta : [0,T] \times \T \longrightarrow \R$ such that $\int_{\T} \zeta(t,\theta) d\theta = 0$ and $||\zeta(t,\cdot)||_{L^2}$ (resp. $||\zeta(t,\cdot)||_{H^{-1}}$) is essentially bounded in $t$ (resp. in $L^2([0,T])$).
\end{definition}

The statement of Theorem~\eqref{thm_lim_hydro} is a quantitative version of the hydrodynamic limit. In~\cite{GORV}, only the error from comparing the microscopic scale to a mesoscopic scale was explicit. This error scaled in~\cite{GORV} like~$\frac{1}{\sqrt{N}}$. \\


\section{Proof of the~Theorem~\ref{thm_lim_hydro}: The two-scale approach}\label{s_two_scale_approach}

For deducing Theorem~\ref{thm_lim_hydro}, we will use the two-scale approach which was invented in~\cite{GORV}. The main idea in the two-scale approach is to introduce an intermediate dynamics on a mesoscopic scale between the microscopic dynamics~(\ref{e_Kawasaki_dynamics}) and the macroscopic dynamics~(\ref{hydro_equn}). The hydrodynamic limit is then deduced in two steps: In the first step, one deduces the convergence of the microscopic dynamics to the mesoscopic dynamics (see Theorem~\ref{p_quantitative_micro_to_meso} from below). In the second step, one deduces the convergence of the mesoscopic dynamics to the macroscopic dynamics (see Theorem~\ref{p_quantitative_meso_to_macro} from below).\\

The most important ingredient in the two-scale approach is the correct definition of the mesoscopic dynamics. The mesoscopic dynamics emerges from projecting the microscopic observables onto mesoscopic observables. The projection onto mesoscopic observables is done with the help of a coarse-graining operator~$P$. We recall that an element $x\in X_N$ is identified with a function on the torus $\mathbb{T} = \mathbb{R}/\mathbb{Z}$ that is piecewise constant with value $x_n$ on $[\frac{n-1}{N},\frac{n}{N})$, $n=1, \ldots N$ (cf.~\eqref{e_embedding_X}). The coarse-graining operator~$P$, that was used in~\cite{GORV}, can be interpreted as the projection of~$X_N$ in $L^2(\mathbb{T})$ onto the space of functions that are piecewise constant on the intervals $\left[ \frac{m-1}{M} , \frac{m}{M} \right)$, $m=1,...,M$. More precisely, this means that first one decomposes the lattice $\left\{ 1, \ldots, N \right\}$ into $M$-many blocks~$B(m)$ of size~$K$ i.e.~$N=MK$ and
\begin{align*}
  B(m)= \left\{ m(K-1)+1, \ldots, mK \right\} \qquad \mbox{for } 1 \leq m \leq M
\end{align*}
Then the operator~$P : X_N \to \mathbb{R}^M $ in~\cite{GORV} is given for~$x \in X_N $ by
\begin{align*}
  P(x) = \left( \frac{1}{K} \sum_{i \in B(1)} x_i, \ldots , \frac{1}{K} \sum_{i \in B(M)} x_i  \right).
\end{align*}

The main difference of this article compared to~\cite{GORV} is that instead the operator~$P$ is defined as the~$L^2$ projection onto splines of order 2 (see Definition~\ref{d_coarse_graining_operator_P} from below). Because spline functions of order 2 are~$C^1(\mathbb{T})$, the mesoscopic variables are more regular compared to~\cite{GORV}. This has two important advantages:
\begin{itemize}
\item In the first step of the two-scale approach, namely showing the convergence of the microscopic dynamics to the mesoscopic dynamics (see Theorem~\ref{p_quantitative_micro_to_meso} below), we get a better error estimate compared to~\cite[Theorem~8]{GORV}. \\[-0.5ex]

\item The second step of the two-scale approach, namely deducing the convergence of the mesoscopic dynamics to the macroscopic dynamics, becomes significantly easier (see Theorem~\ref{p_quantitative_meso_to_macro} from below). Instead of a mixed method we can apply a direct Galerkin approximation method.
\end{itemize}

However, there is a trade-off compared to the argument of~\cite{GORV}. For deducing the convergence of the microscopic dynamics to the mesoscopic dynamics (see Proposition~\ref{p_quantitative_micro_to_meso}) one needs certain ingredients, among them is a uniform logarithmic Sobolev inequality (LSI) and the strict convexity of the coarse-grained Hamiltonian. Deducing those ingredients becomes significantly more difficult compared to~\cite{GORV}. \\

We provide those ingredients and several other technical results in the companion article~\cite{DMOW18b}. The uniform LSI and the strict convexity of the coarse-grained Hamiltonian were originally provided in Deniz Dizdar's diploma thesis~\cite{Dizdar}. The main estimate to deduce the convergence of the microscopic dynamics to the mesoscopic dynamics (see Theorem~\ref{centralestimate} from below) also was deduced in Dizdar's diploma thesis.

\begin{remark} In the companion article ~\cite{DMOW18b}, we work with functions with unrestricted mean. In this present article, we work with functions with mean $0$. Although this causes slight difference in the definitions, the results there apply here with minimal changes. 
\end{remark}

Let us now turn to the definition of the mesoscopic dynamics.
\begin{definition}[Definition of the coarse-graining operator~$P$]\label{d_coarse_graining_operator_P}
  For $M \in \mathbb{N}$, let $Y=Y_M$ be the space of spline functions of degree~$2$ with mean 0 on the torus $\mathbb{T}=[0,1]$ corresponding to the mesh $\left\{\frac{m}{M}\right\}_{m \in [M]}$. That is
\begin{align*}
  Y_M := & \left\{y\in C^{1}(\mathbb{T})|\,\forall m\in [M]:   y|_{\left(\frac{m-1}{M},\frac{m}{M}\right)}\, \text{polynomial of degree $\leq$ 2}, \right. \\
  & \left. \mbox{ and} \int_0^1 y(\theta) d\theta = 0\right\}.
\end{align*}
We endow $Y_M$ with the inner product inherited from $L^2(\mathbb{T})$. We define the coarse graining operator $P: L^2(\mathbb{T})\rightarrow Y_M\subset L^{2}(\mathbb{T})$ as the $L^2$-orthogonal projection onto $Y_M \subset L^2(\mathbb{T})$. 

\end{definition}

From now on, we assume $N = KM$ for $K\in \mathbb{N}$.

\begin{definition}[Two notions of adjoints to the coarse-grain operator $P$]\label{d_adjoint_NPt} Restrict the coarse-grain operator $P$ to $P: X_N \rightarrow Y_M$. First, define map~$P^{t}: Y_M  \to X_N$ as the adjoint to $P$, when we endow $X_N$ with the Euclidean inner product (viewing $X_N$ as a subspace of $\mathbb{R}^N$), i.e. for all~$x \in X_N$, $y \in Y_M$ 
\begin{align*}
  \langle Px, y \rangle_{L^{2}} =  x \cdot P^{t} y.
\end{align*}
It follows that the map $NP^t:Y_M  \to X_N$ is the adjoint to $P$, when we endow $X_N$ with the $L^2$ inner product (viewing $X_N$ as a subspace of $L^2(\mathbb{T})$), i.e. for all~$x \in X_N$, $y \in Y_M$ 
\begin{align*}
\langle Px, y \rangle_{L^{2}} =  \langle x, NP^t y \rangle_{L^{2}}.
\end{align*}
From this it also follows the map $NP^t$ is the $L^2$-orthogonal projection of $Y_M$ onto $X_N$, explicitly given by
\begin{align}\label{e_adjoint_NP^t_explicit}
\left( NP^t y \right)_i = N \int_{\frac{i-1}{N}}^{\frac{i}{N}} y(\theta) d \theta \quad \mbox{for} \quad i \in \{1, 2, \cdots, N\}.
\end{align}
\end{definition}

We recall a lemma from \cite{DMOW18b} (cf. Lemma 3.12 in \cite{DMOW18b}).

\begin{lemma}\label{p_invertible_PNPt}
	It holds that
	\begin{align}\label{e_p_operator_estimate}
	\| PNP^t  - \id_{Y_M} \| = O\left(\frac{1}{K^2}\right).
	\end{align}
	In particular, if~$K = \frac{N}{M}$ is large enough, then $P NP^t: Y_M \rightarrow Y_M$ is invertible. 
\end{lemma}

From now on, we will assume $K$ is sufficiently large so that $P NP^t: Y_M \rightarrow Y_M$ is invertible. In particular, this means $P: X_N \rightarrow Y_M$ is onto and $NP^t: Y_M \rightarrow X_N$ is one-to-one. 

\begin{definition}[Coarse-grained Hamiltonian~$\bar H$.]
  The coarse grained Hamiltonian~$\bar H: Y_M \to \mathbb{R}$ is given by 
  \begin{align}\label{e_def_coarse_grained_Hmailtonian}
   \bar H(y) = - \frac{1}{N} \log \int_{\left\{x \in X_N  :  Px =y \right\}} \exp \left( - H(x) \right) \mathcal{L}^{N-M}(dx),
  \end{align}
where~$\mathcal{L}^{N-M}$ denotes the $N-M$-dimensional Hausdorff measure.
\end{definition}

\begin{definition}[Mesoscopic dynamics]
The mesoscopic dynamics~$\eta$ is given by a solution of the ordinary differential equation
  \begin{align}
    \label{e_def_mesoscopic_dynamics}
    \frac{d}{dt} \eta(t) = - \bar A  \nabla \bar H(\eta(t)),
  \end{align}
where 
\begin{align}
  \label{e_definition_bar_A}
  \bar A : = P A NP^t.
\end{align}
\end{definition}
\begin{remark}
  Since $A$ is positive definite on $X_N$, $P^t: Y_M \to X_N$ is one-to-one, and $P: X_N \to Y_M$ is the adjoint to $P^t$, we see that $\bar{A}$ is positive definite on $Y_M$. 
\end{remark}
\begin{remark}\label{r_remark_choosing_L_2}
 In this work we consider splines of order~$L=2$, because then the operator~$AP^t \bar A^{-1}$ is bounded (see Lemma~\ref{operator} below). If one chooses splines of lower order then the operator~$AP^t \bar A^{-1}$ is unbounded. In~\cite{GORV}, the coarse-graining operator~$P$ was defined as the~$L^2$-orthogonal projection onto piecewise constant functions i.e.~splines of order~0. In~\cite{GORV}, one worked around the problem that operator~$AP^t \bar A^{-1}$ is unbounded by using a less straight-forward definition of~$\bar A$ as $\bar A^{-1} :=P A^{-1}N P^t$. That choice lead to a sub-optimal error when comparing the microscopic to the mesoscopic evolution (see also Remark~\ref{optimalsc} below). Choosing~$L>2$ does not improve the error derived with our method further. 
\end{remark}

Now, we state the first ingredient of the two-scale approach.
  \begin{theorem}[Convergence of the microscopic to the mesoscopic dynamics]\label{p_quantitative_micro_to_meso}
Under the same assumption as in Theorem~\eqref{hydro_equn}, let~$f \mu$ denote the distribution of the Kawasaki dynamics~$X_t$ (cf.~Lemma~\ref{p_fokker_planck_kawasaki}) and let~$\eta$ denote the solution of the mesoscopic equation~\eqref{e_def_mesoscopic_dynamics}. Then
\begin{align}
 \sup_{0\leq t\leq T}  \int | x-\eta(t) |_{H^{-1}}^2 \,f (t,x) \mu(dx) & \lesssim    \int |Px - \eta(0)|_{H^{-1}}^2 f(0,x) \mu (dx) \\
& \ +  \frac{T}{K}   + \frac{1}{M^2}  \left(C_{\Ent}   +1  \right).  \label{e_micro_to_meso_sup}
\end{align}
where $C_{\Ent}$ is given by \eqref{e_closness_initial_data}.
  \end{theorem}
We give the proof of Theorem~\ref{p_quantitative_micro_to_meso} in Section~\ref{s_micro_to_meso}. The error term~$\frac{T}{K} $ on the right hand side of~\eqref{e_micro_to_meso_sup} comes from the fact that we compare a stochastic microscopic dynamic to a deterministic mesoscopic dynamic. The scaling corresponds to what one would expect from the central limit theorem, if we had chosen $L = 0$. In that case, $y$ is a vector whose entries are means of $K$ weakly correlated random variables and $\eta$ is interpreted as the vector whose entries are expected values of these means. Now, let us state the second ingredient in the two-scale approach.

  \begin{theorem}[Convergence of the mesoscopic to the macroscopic dynamics]\label{p_quantitative_meso_to_macro}
Let~$\eta$ denote the solution of the mesoscopic dynamics~\eqref{e_def_mesoscopic_dynamics} and let~$\zeta$ denote the solution of the macroscopic dynamics~\eqref{hydro_equn}. Then 
    \begin{align}
      \label{e_quantitative_meso_to_macro}
       \sup_{0\leq t\leq T} & |\zeta(t) - \eta(t) |_{H^{-1}}^2 + \int_{0}^T |\zeta(s)- \eta(s)|_{L^2}^2 ds  \\
& \lesssim |\zeta(0) - \eta(0) |_{H^{-1}}^2 +  \frac{T}{K^2} +  \left(\frac{1}{K^2} + \frac{1}{M^2}\right) |\zeta(0)|_{L^2}^2 .
    \end{align}
  \end{theorem}

We prove Theorem~\ref{p_quantitative_meso_to_macro} in Section~\ref{s_meso_macro}. For the proof we adapt a standard method from numerical analysis. The mesoscopic evolution~\eqref{e_def_mesoscopic_dynamics} is interpreted as a Galerkin approximation of the macroscopic evolution~\eqref{hydro_equn}. The non-standard part of the argument is that when comparing~\eqref{e_def_mesoscopic_dynamics} to~\eqref{hydro_equn} one gets two additional error terms. One error term comes from approximating the Euclidean structure~$\langle \cdot, \cdot \rangle_{H^{-1}}$ by the Euclidean structure~$ \langle \cdot ,  \bar A^{-1} \cdot \rangle_{L^2}$ and the second error terms comes from approximating the gradient of the macroscopic free energy~$\mathcal{H}(\zeta):= \int_{\mathbb{T}} \varphi (\zeta (\theta)) d \theta$ by the gradient of the coarse-grained Hamiltonian~$\bar H$.\\

Beside Theorem~\ref{p_quantitative_micro_to_meso} and Theorem~\ref{p_quantitative_meso_to_macro}, the only ingredients of the proof of the main result (cf. Theorem \ref{thm_lim_hydro}) are some basic facts about splines.
\begin{lemma}\label{p_facts_spline_approximation}
 Let $P: L^{2}(\mathbb{T}) \to Y_M$ denote the $L^2$-orthogonal projection onto the spline space~$Y_M \subset L^2(\mathbb{T})$. It holds that for any function~$\zeta \in L^2$
 \begin{align}
   & |\zeta - P\zeta|_{H^{-1}}  \lesssim \frac{1}{M} |\zeta-P \zeta|_{L^2} \lesssim \frac{1}{M^2} |\zeta|_{H^1} , \qquad  \mbox{and}   \label{e_approximation_H_neg_1} \\
  & |P \zeta|_{H^1} \lesssim |\zeta|_{H^1}, \qquad |P \zeta|_{H^{-1}} \lesssim |\zeta|_{H^{-1}}.   \label{e_operator_norm_bounded_H_neg}
 \end{align}
As a result, we can extend $P$ to an operator $P: H^{-1}(\mathbb T) \rightarrow Y_M$ such that for any $\zeta \in H^{-1}, \xi \in Y_M$,
\begin{align}
\langle P\zeta, \xi \rangle_{L^2} = \langle \zeta, \xi \rangle_{L^2}.
\end{align} 
\end{lemma}
The second estimate of \eqref{e_approximation_H_neg_1} and the first estimate of \eqref{e_operator_norm_bounded_H_neg} were deduced in the companion article~\cite{DMOW18b}; the other two estimates from these two by simple duality arguments. We are now ready to give the proof of Theorem ~\ref{thm_lim_hydro}.

\begin{proof}[Proof of Theorem~\ref{thm_lim_hydro}]
We choose as the initial condition of the mesoscopic dynamics~$\eta$ given by~\eqref{e_def_mesoscopic_dynamics} the function~$\eta(t=0) = P \zeta(t=0)$. Applying the triangle inequality, Theorem~\ref{p_quantitative_micro_to_meso}, Theorem~\ref{p_quantitative_meso_to_macro},~$| P\zeta(0) |_{L^2}^2 \leq  | \zeta(0) |_{L^2}^2$ and~$N=KM$ yields the estimate
\begin{align}
&   \sup_{0\leq t\leq T}   \int | x- \zeta(t) |_{H^{-1}}^2 \,f (t,x) \mu(dx) \\
& \quad \leq \sup_{0\leq t\leq T}  \int 2| x- \eta(t) |_{H^{-1}}^2 \,f (t,x) \mu(dx)  + \sup_{0\leq t\leq T}  2| \eta(t)-\zeta(t) |_{H^{-1}}^2 \\
& \quad \lesssim    \int |Px - P\zeta(0)|_{H^{-1}}^2 f(0,x) \mu (dx)  + |\zeta(0) - P\zeta(0) |_{H^{-1}}^2 \\
&\qquad   +  \frac{T}{K} + \frac{T}{K^2} + \frac{1}{M^2}  (C_{\Ent} + 1) + \left(\frac{1}{K^2} + \frac{1}{M^2}\right) | \zeta(0) |_{L^2}^2.
\end{align}
Applying~\eqref{e_approximation_H_neg_1} and~\eqref{e_operator_norm_bounded_H_neg}, and choosing~$K=M^2$ yields the desired estimate~\eqref{e_quantitative_estimate_1}. 
\end{proof}

\section{Proof of Theorem~\ref{p_quantitative_micro_to_meso}: Convergence of microscopic dynamics to mesoscopic dynamics} \label{s_micro_to_meso}

The proof of Theorem~\ref{p_quantitative_micro_to_meso} is quite complex. Before proceeding to the rigorous argument let us give some heuristics. Theorem~\ref{p_quantitative_micro_to_meso} states that the stochastic microscopic evolution given by the Kawasaki dynamics (see~\eqref{e_Kawasaki_dynamics} ), i.e.
\begin{align}
  dX_t = - A \nabla H (X_t) dt + \sqrt{2A} dB_t,
\end{align}
is close in the~$H^{-1}-$norm to the mesoscopic deterministic dynamics given by~\eqref{e_def_mesoscopic_dynamics} i.e.
\begin{align}\label{e_def_mesoscopic_dynamics_2}
  \frac{d}{dt} \eta = - \bar A  \nabla \bar H(\eta).
\end{align}
The first observation needed is that because the~$H^{-1}-$norm is a weak norm (i.e.~it involves integration, see Definition~\ref{d_h_neg_norm}) one can control the error between~$X_t$ and the projected process~$PX_t$ (cf.~also Lemma~\ref{p_facts_spline_approximation}). Hence, it suffices to show that the stochastic  evolution
\begin{align}\label{e_projected_Kawasaki}
  dP X_t = - P A \nabla H (X_t) dt + P \sqrt{2A} dB_t
\end{align}
is close to the deterministic mesoscopic dynamics~\eqref{e_def_mesoscopic_dynamics_2}. Because the operator~$P$ takes averages over blocks of size~$K$, the noise term~$P \sqrt{2A} dB_t$ of the projected Kawasaki dynamics~\eqref{e_projected_Kawasaki} should vanish as~$K\to \infty$ by the law of large numbers. It is left to show that 
\begin{align}\label{e_projected_deterministic_Kawasaki}
  \frac{d}{dt} P X_t = - P A \nabla H (X_t) 
\end{align}
is close to the mesoscopic dynamics~\eqref{e_def_mesoscopic_dynamics_2}. By definitions and a short calculation one sees that the mesoscopic dynamics~\eqref{e_def_mesoscopic_dynamics_2} is given by
\begin{align}  \label{e_meso_written_out}
  \frac{d}{dt} \eta (t) = - P A\  \mathbb{E}_{\mu}\left[  \nabla H (x) \ | \ Px = \eta(t) \right],
\end{align}
where the expectation is taking with respect to the canonical ensemble~$\mu$ conditioned on the mesoscopic profile given by~$\eta(t)$. We observe that~$\mu$ is also the stationary distribution of the Kawasaki dynamics~\eqref{e_Kawasaki_dynamics} (see also Lemma~\ref{p_fokker_planck_kawasaki}). The process~$X_t$ equilibrates a lot faster on blocks of size~$K$ than in the whole system. Hence, we expect that the dynamics~\eqref{e_projected_deterministic_Kawasaki} and~\eqref{e_meso_written_out} are close if the blocks are a lot smaller compared to the overall system size~$N$, in other words~$\frac{K}{N} \to 0$. In the rigorous argument, this fact will be quantified with the help of a uniform LSI which characterizes the speed of the convergence to equilibrium (see Theorem~\ref{p_LSI_conditional} below).\\

Let us turn now to the rigorous proof of Theorem~\ref{p_quantitative_micro_to_meso}. The first ingredient of the proof is an estimate of the second moment of $X_t$ in $L^2$ norm, which controls the difference between~ $X_t$ and the projected dynamics ~$PX_t$ in $H^{-1}$ norm by Lemma~\ref{p_facts_spline_approximation}.
\begin{proposition}\label{p_moment_estimate}
	Let~$f(t)\mu$ denote the law of the Kawasaki dynamics~$X_t$ (cf.~Lemma~\ref{p_fokker_planck_kawasaki}). Then it holds that
	\begin{align}
	\label{e_moment_estiamte}
	\ \int |x|^2 \,f\, \mu (dx) & \lesssim NC_{\Ent} + \int |x|^2 \mu (dx) \lesssim N (C_{\Ent} + 1).
	\end{align}
\end{proposition}
For the proof of Proposition~\ref{p_moment_estimate} we refer to~\cite{GORV} (It is Proposition~24 in~\cite{GORV}.). We note that the second estimate in~\eqref{e_moment_estiamte} follows directly from definition~\eqref{e_def_Gibbs_measure} i.e.~$ \int |x|^2 \mu (dx) \lesssim  N$ (see also~(88) in~\cite{GORV}).\\

The next ingredient of the proof is the equivalence of $H^{-1}$ norm with a more naturally defined norm on $Y_M$ that comes from the inner product induced by the positive definite operator ~$\bar{A}^{-1}$, namely
\[
\langle y, z \rangle_{\bar{A}^{-1}} = \langle y , \bar{A}^{-1} z \rangle_{L^2} \quad \mbox{and} \quad |y|_{\bar{A}^{-1}} = \sqrt{\langle y , \bar{A}^{-1} y \rangle_{L^2}}.
\]

\begin{lemma} \label{p_equivalence_neg_H_norm_and_norm_generated_by_neg_bar_A} 
 There exists an integer $K^*$ such that for all $K \geq K^*, M$ and all~$y \in Y_M$, 
\begin{align}
  \label{e_equivalence_neg_H_norm_and_neg_bar_A}
  |y|_{\bar{A}^{-1}} \simeq |y|_{H^{-1}}.
\end{align}
\end{lemma}
We will deduce Lemma~\ref{p_equivalence_neg_H_norm_and_norm_generated_by_neg_bar_A} in Section~\ref{s_spline_approximations}, where we gather and prove facts about splines. \\

The last (and main) ingredient for the proof is the following estimate which controls the difference between the projected microscopic dynamics $PX_t$ and the mesoscopic dynamics $Y_t$ in $\bar{A}^{-1}$ norm (hence in $H^{-1}$ norm by Lemma \ref{p_equivalence_neg_H_norm_and_norm_generated_by_neg_bar_A}).
\begin{theorem} \label{centralestimate} 
Under the same assumptions as in Theorem~\ref{p_quantitative_micro_to_meso}, there is an integer~$K^*$ and~$\lambda >0$ such that for all~$K \geq K^*$ and any finite time $T>0$ it holds
\begin{align}
\sup_{0\leq t\leq T} & \int \frac{1}{2} |P x-\eta|^2_{\bar{A}^{-1}} \,f\,
\mu(dx)  +  \lambda \int_0^T dt \int |Px-\eta|_{L^2}^2 \, f\, \mu(dx) \nonumber\\
&\leq \int  \,|P x-\eta(0)|^2_{\bar{A}^{-1}} \,f(0)\, \mu(dx)  + \frac{2T}{K} + 2C \ \frac{C_{\Ent}}{M^2} .\label{centest1}
\end{align}
\end{theorem}
 \begin{remark}
The estimate~\eqref{centest1} also shows that the projected Kawasaki dynamics~\eqref{e_projected_Kawasaki} is close to the mesoscopic dynamics~\eqref{e_def_mesoscopic_dynamics_2} using a time-integrated strong norm. This is reminiscent of the well-known phenomenon of parabolic improvement in numerical analysis. 
\end{remark}
\begin{remark} \label{r_explicit_constant}
  The universal constant~$0<C< \infty$ in Theorem~\ref{centralestimate} is given by $C = \displaystyle\frac{\kappa^2\gamma}{4\sigma^2\lambda\varrho^2}$, where the constants~$\kappa$,~$\lambda$,~$\gamma$, $\sigma$, and~$\varrho$ are given by: \newline  $\mbox{}  \ \cdot    \kappa:= \|\operatorname{Hess} H\|$, which is bounded independently of $N$ by the assumption~\eqref{hamiltonian},~\eqref{e_structure_of_ss_potential} and~\eqref{e_assumptions_on_ss_perturbation};\newline $\mbox{}  \ \cdot   2\lambda$ the lower bound on~$\Hess \bar{H}$ as in Theorem~\ref{p_strict_convexity_coarse_grained_Hamiltonian}; \newline
$\mbox{}  \ \cdot  \varrho$ is the constant of the logarithmic Sobolev inequality (LSI) from
Theorem~\ref{p_LSI_conditional}  from below;\newline
$\mbox{}  \ \cdot   \sigma$ is the constant from Lemma \ref{operator}; \newline
$\mbox{}  \ \cdot   \gamma$ the constant from Lemma \ref{Poin} below.
\end{remark}

\begin{remark} \label{optimalsc}
Theorem \ref{centralestimate} was first derived in Dizdar's diploma thesis~\cite{Dizdar}. Theorem \ref{centralestimate} should be compared with Theorem 8 in~\cite{GORV}. They arrive at a similar bound for the deviation from hydrodynamic behavior with additional term, scaling $M^{-1}$. As mentioned before this additional error term occurs due to their choice of the coarse-graining operator~$P$ as the projection onto piecewise constant functions and the different definition of~$\bar A$. 
\end{remark}

We will prove ~Theorem~\ref{centralestimate} in Section~\ref{s_central_estimate} and finish this section with a quick derivation of Theorem~\ref{p_quantitative_micro_to_meso} based on the ingredients above. 
\begin{proof}[Proof of Theorem~\ref{p_quantitative_micro_to_meso}] Using the triangle inequality,
\begin{align}
 \int  | x - \eta |_{H^{-1}}^2\,f \,\mu(dx)  
  \leq \int 2 | x - Px |_{H^{-1}}^2 \,f \, \mu(dx)  + \int 2 | Px - \eta |_{H^{-1}}^2 \,f\, \mu (dx) .
\end{align}
The first term on the right hand side is estimated by Lemma~\ref{p_facts_spline_approximation} and Proposition \ref{p_moment_estimate}. The second term on the right hand side is estimated by Theorem~\ref{centralestimate} and Lemma~\ref{p_equivalence_neg_H_norm_and_norm_generated_by_neg_bar_A}. This verifies the estimate~\eqref{e_micro_to_meso_sup}. 
\end{proof}

\subsection{Proof of Theorem~\ref{centralestimate}} \label{s_central_estimate}

The proof of Theorem~\ref{centralestimate} is based on several auxiliary statements. Those auxiliary statements will be deduced in Section~\ref{s_micro_to_meso_proof_aux_results}. In the proof of Theorem~\ref{centralestimate} we need to disintegrate the canonical ensemble~$\mu$ given by~\eqref{e_def_Gibbs_measure} into a conditional measure~$\mu (dx | P x=y)$ and the marginal~$\bar \mu (dy)$. 
\begin{definition}[Disintegration of the canonical ensemble~$\mu$]
  The coarse-graining operator~$P: L^2 \to Y_M$ introduces a decomposition of the canonical ensemble~$\mu$ into conditional measure~$\mu(dx| Px=y)$ and marginal measures~$\bar \mu(dy)$. More precisely the measures $\mu(dx| Px=y)$ and~$\bar \mu(dy)$ are defined by the relation
  \begin{align*}
    \int f (x) \mu (dx) = \int \int f(x) \mu(dx| Px=y) \bar \mu(dy) 
  \end{align*}
for any test function~$f$. This means that the conditional measure~$\mu(dx| Px=y)$ is a probability measure on the space
\begin{align*}
 \left\{ x \in X_N \ | \ Px=y\right\} \subset X_N
\end{align*}
that is absolutely continuous wrt.~the~$N-M$ dimensional Hausdorf measure~$\mathcal{L}^{N-M}$. Its Radon-Nikodym derivative is given by
\begin{align}
  \label{e_d_conditional_can_ensemble}
  \frac{ \mu(dx | Px =y )}{d \mathcal{L}^{N-M}} (x) = \frac{1}{Z} \ \mathds{1}_{\left\{ Px=y \right\}}(x) \ \exp (- H(x)  ).
\end{align}
For convenience, we also may write~$\mu(dx|y)$ instead of~$\mu(dx|Px=y)$.

The marginal measure~$\bar \mu$ is a probability measure on the space $Y_M$ that is absolutely continuous wrt.~the $M-1$-dimensional Hausdorf measure~$\mathcal{L}^{M-1}$. Its Radon-Nikodym derivative is given by
\begin{align}
  \label{e_d_marginal}
  \frac{d \bar \mu}{d \mathcal{L}^{M-1}} (y) = \frac{1}{Z}  \exp\left( -N \bar H(y)  \right),
\end{align}
where~$\bar H$ is the coarse-grained Hamiltonian given by~\eqref{e_def_coarse_grained_Hmailtonian}.
\end{definition}

Starting point of the proof of Theorem~\ref{centralestimate} is the following formula.
\begin{lemma}\label{p_curcialestimate_crucial_formula}
  For a function~$f: X_N \to \mathbb{R}$ and~$y \in Y_M$ let~$\bar f(y)$ denote
\begin{align}
  \label{e_d_bar_f}
  \bar f (y) = \int f(x) \mu(dx|y).
\end{align}
Then  holds that 
  \begin{align}
    \frac{d}{dt} \int  \frac{1}{2} | P x-\eta|^2_{\bar{A}^{-1}} f \mu (dx) &= \frac{\operatorname{dim}Y_M}{N}
         - \int \langle y-\eta, \nabla \bar{H}(y) - \nabla \bar{H}(\eta) \rangle_{L^2} \bar{f} \bar{\mu} (dy)\\
    &\quad -\int   AP^t\bar{A}^{-1} (y-\eta)  \cdot \operatorname{cov}_{\mu(dx|y)}(f,\nabla H) 
 \bar{\mu} (dy). \label{e_curcialestimate_crucial_formula}
  \end{align}
\end{lemma}
We will deduce Lemma~\ref{p_curcialestimate_crucial_formula} in Section~\ref{s_aux_results_micro_to_meso}. Let us have a closer look at the formula~\eqref{e_curcialestimate_crucial_formula}. The first term of the right hand side, since $N^{-1}M = K^{-1}$ and~$\dim Y_M=M-1$, has the scaling that could be expected from
the central limit theorem. Therefore this error term estimates the discrepancy that the Kawasaki dynamics~\eqref{e_Kawasaki_dynamics} has noise whereas the mesoscopic dynamics~\eqref{e_def_mesoscopic_dynamics} is deterministic.\\

Let us have a look at the two remaining terms on the right hand side of~\eqref{e_curcialestimate_crucial_formula}. The second term on the right hand side is a good term because of the uniform convexity
of $\bar{H}$.
\begin{theorem}[Strict convexity of~$\bar H$]\label{p_strict_convexity_coarse_grained_Hamiltonian}
There are constants~$0< \lambda, \Lambda, K^* < \infty$ such that for all~$K \geq K^*$,~$M$ and all~$y \in Y_M$ it holds
\begin{align*}
2 \lambda \Id_{Y_M} \leq  \Hess \bar H(y) \leq 2 \Lambda \Id_{Y_M}
\end{align*}
in the sense of quadratic forms.
\end{theorem}
The proof of Theorem~\ref{p_strict_convexity_coarse_grained_Hamiltonian} is quite complex. It is deduced in the companion article (see Theorem~1.6 in~\cite{DMOW18b}). With the help of Theorem~\ref{p_strict_convexity_coarse_grained_Hamiltonian}, the following estimate of the second term on the right hand side of~\eqref{e_curcialestimate_crucial_formula} is immediate.
\begin{corollary}\label{p_centralestimate_convexity_applied}
  It holds that
\begin{equation}
\int \langle y-\eta, \nabla \bar{H}(y) - \nabla \bar{H}(\eta) \rangle_{L^2} \bar{f}\, \bar{\mu} (dy)
\geq 2\lambda\int  |y-\eta |_{L^2}^2 \bar{f}\, \bar{\mu} (dy). \label{e_centralestimate_convexity_applied}
\end{equation}
\end{corollary}

Let us turn to the estimation of the third term on the right hand side of~\eqref{e_curcialestimate_crucial_formula}. We have to
deal both with the operator~$AP^t\bar{A}^{-1}$ and with the covariance~$\operatorname{cov}_{\mu(dx|y)}(f,\nabla H) $. The operator~$AP^t\bar{A}^{-1}$ measures the non-commutativity of projecting and taking second differences. It would favor macroscopic
description through (non-local) low-frequency Fourier modes. It is easy to check that its operator norm blows up as
$N\rightarrow \infty$ if one projects on piecewise constant or piecewise linear functions (i.e.~projection on splines of zero or first order). However, we do get a good control if we project on splines of second order: 

\begin{lemma} \label{operator}
There exists a universal constant $\sigma > 0$ and an integer $K^*$ such that for all $K \geq K^*, M$ and all $y\in Y_M$ it holds
\begin{align}
   |ANP^t \bar A^{-1} y  |_{L^2} \leq  \frac{1}{\sigma} |y |_{L^2} \label{opestimate}.
\end{align}
\end{lemma}
The proof of Lemma~\ref{operator} is given in Section~\ref{s_spline_approximations}, where we gather and prove facts about splines. Let us now turn to the covariance term on the right hand side of~\eqref{e_curcialestimate_crucial_formula}. It is estimated by the next lemma.
\begin{lemma}[Covariance estimate]\label{p_crucial_covariance_estimate}
There is a universal constant~$0<C_{\cov}< \infty$ such that
  \begin{align}
    \label{e_crucial_covariance_estimate}
    \int  \frac{ | \cov_{\mu(dx|y)} \left(f , \nabla H \right)|^{2}}{\bar f} \bar \mu (dy) \leq \frac{C_{\cov}}{M^2} \int \frac{\nabla f \cdot A \nabla f}{f} \mu(dx).
  \end{align}
\end{lemma}
\begin{remark}
 The universal constant~$C_{\cov}$ in the estimate~\eqref{e_crucial_covariance_estimate} is given by~$   \displaystyle\frac{\kappa^2 \gamma}{\varrho^2}$ 
where the constants~$\kappa$,~$\gamma$ and~$\varrho$ are given in Remark~\ref{r_explicit_constant}.
\end{remark}
The proof of Lemma~\ref{p_crucial_covariance_estimate} is given in Section~\ref{s_aux_results_micro_to_meso}. It is based on two facts. The first one is highly nontrivial: the measures $\mu(dx|y)$ satisfy a uniform LSI (see Theorem~\ref{p_LSI_conditional} below). This fact is deduced in the companion article~\cite{DMOW18b}. The second fact is the strong penalization of fluctuations around macroscopic observables in $L^2$ by the norm associated
with the positive definite matrix $A$ (see Lemma \ref{Poin} below).
\begin{proof}[Proof of Theorem~\ref{centralestimate}]
  Applying Lemma~\ref{p_curcialestimate_crucial_formula} and Corollary~\ref{p_centralestimate_convexity_applied} yields that 
  \begin{align}
        \frac{d}{dt}  \int & \frac{1}{2} | Px-\eta|^2_{\bar{A}^{-1}} f\, \mu (dx) + 2 \lambda \int |y - \eta|_{L^2}^2 \bar f \bar \mu (dy) \\
     &\leq \frac{\operatorname{dim}Y_M}{N}
    -\,\int   AP^t\bar{A}^{-1} (y-\eta)  \cdot \operatorname{cov}_{\mu(dx|y)}(f,\nabla H) 
 \,\bar{\mu} (dy)
   \label{e_p_centralestimate_crucial_formula}
  \end{align}
Applying Lemma~\ref{operator}, Lemma~\ref{p_crucial_covariance_estimate}, and Young's inequality yields that
\begin{align}
&  \left| \int   AP^t\bar{A}^{-1} (y-\eta)  \cdot \operatorname{cov}_{\mu(dx|y)}(f,\nabla H) 
 \,\bar{\mu} (dy) \right| \\
& \quad \leq \left( \int   |ANP^t\bar{A}^{-1} (y-\eta)|_{L^2}^2 \bar f  \bar{\mu} (dy) \right)^{\frac{1}{2}}  \left(  \int \frac{|\operatorname{cov}_{\mu(dx|y)}(f,\nabla H)|_{L^2}^2}{\bar f} 
 \,\bar{\mu} (dy) \right)^{\frac{1}{2}} \\
& \quad \leq  \left( \int \frac{1}{\sigma^2} | y-\eta|_{L^2}^2 \bar f  \bar{\mu} (dy) \right)^{\frac{1}{2}} \left(  \frac{C_{\cov}}{N M^2}\int \frac{\nabla f \cdot A \nabla f}{f} \mu(dx) \right)^{\frac{1}{2}}\\
& \quad \leq \lambda  \int  | y-\eta|_{L^2}^2 \bar f  \bar{\mu} (dy)  + \frac{C_{\cov}}{4 \lambda \sigma^2 M^2}   \frac{1}{N} \int \frac{\nabla f \cdot A \nabla f}{f} \mu(dx)\\
& \quad = \lambda  \int  | y-\eta|_{L^2}^2 \bar f  \bar{\mu} (dy) - \frac{C_{\cov}}{4 \lambda \sigma^2 M^2}  \frac{1}{N} \ \frac{d}{dt} \Ent \left(f(t) \mu | \mu \right), \label{e_p_centralestimate_estiamtion_second_term}
\end{align}
where we used in the last step the observation that
\begin{align}
  \label{e_time_derivative_entropy}
  \frac{d}{dt} \Ent \left(f(t) \mu | \mu \right) = -  \int \frac{\nabla f \cdot A \nabla f }{f} \mu(dx)
\end{align}
Combining~\eqref{e_p_centralestimate_crucial_formula} and~\eqref{e_p_centralestimate_estiamtion_second_term} and integrating over the time interval~$[0,T]$ yields the desired estimate~\eqref{centest1}.
\end{proof}

\section{Proof of Theorem~\ref{p_quantitative_meso_to_macro}: Convergence of mesoscopic dynamics to macroscopic dynamics}\label{s_meso_macro}

In this section we state the proof of Theorem~\ref{p_quantitative_meso_to_macro}. We need to show that the mesoscopic evolution~\eqref{e_def_mesoscopic_dynamics}
\begin{align}
      \frac{d}{dt} \eta(t) = - \bar A  \nabla \bar H(\eta(t))
\end{align}
converges to the macroscopic evolution~\eqref{hydro_equn} i.e.
\begin{align}
  \frac{ \partial }{\partial t} \zeta (t) =\frac{\partial^2}{\partial \theta^2}\nabla \mathcal{H} (\zeta (t)) = \frac{\partial^2}{\partial \theta^2}\varphi'(\zeta (t)).
\end{align}

Formally, this means that one has to exchange the operator~$\bar A$ with the operator~$\frac{\partial^2}{\partial \theta^2}$ and the functional ~$\nabla \bar H(\cdot) $ with the functional ~$\nabla \mathcal{H}(\cdot)$. This sounds plausible because~$\bar A  = PNAP^{t}$ involves the second order difference operator~$A$ and, as is shown in the companion article~\cite{DMOW18b}, the gradient of coarse-grained Hamiltonian, $\nabla \bar H (\cdot)$, converges to the gradient of macroscopic free energy, $\nabla \mathcal{H} (\cdot)$. \\

The argument of Theorem~\ref{p_quantitative_meso_to_macro} is inspired by the Galerkin approximation scheme, which is a well-known method in numerical analysis, based on some auxiliary results. The first auxiliary result is the strict convexity of the coarse grained Hamiltonian~$\bar H$ (cf.~Theorem~\ref{p_strict_convexity_coarse_grained_Hamiltonian}). We also need that the  macroscopic free energy~$\mathcal{H}$ is strictly convex:
\begin{lemma}[Strict convexity of the macroscopic free energy~$\mathcal{H}$]\label{p_strict_convexity_varphi}
  The function~$\varphi: \mathbb{R} \to \mathbb{R}$ given by~\eqref{average_hydro_potential} is smooth and satisfies
\begin{align}
 \varphi' (0)=0 \quad \mbox{and}  \quad  0 < \lambda  \leq  \varphi''(\theta) \leq \Lambda < \infty \quad \mbox{for all~$\theta \in \mathbb{R}$.} \label{e_convexity_estimate_local_free_energy}
\end{align} 
\end{lemma}
We omit the proof of Lemma~\ref{p_strict_convexity_varphi}, which follows from basic estimates and properties of the Legendre transform (see for example~\cite[Lemma 41]{GORV}) and assumption \eqref{e_wlog_mean_0}. The next auxiliary result is that the gradients of the free energies~$\bar H$ and of~$\mathcal{H}$ are close:
\begin{lemma}[Convergence of gradient of coarse-grained Hamiltonian to gradient of macroscopic free energy] \label{p_convergence_meso_to_macro_free_energy}
There is an integer~$K^*$ such that if~$K \geq K^*$ then it holds for all~$x \in L^2(\mathbb{T})$
  \begin{align}
    \label{e_estimate_meso_to_macro_free_energy}
    \left|  \nabla \bar H (Px) -   \nabla \mathcal{H}(x)  \right|_{L^2}  \lesssim  \left(\frac{1}{K}+\frac{1}{M}\right)  |x|_{H^{1}} + \frac{1}{K} .
  \end{align}
\end{lemma}
Lemma~\ref{p_convergence_meso_to_macro_free_energy} is deduced in the companion article (see Theorem~1.9 in~\cite{DMOW18b}). The next auxiliary result provides a-priori energy estimates for the proof of Theorem~\ref{p_quantitative_meso_to_macro}.
\begin{lemma}\label{p_aux_estimates_meso_to_macro}
Let~$\zeta(t)$ denote the macroscopic dynamics given by~\eqref{e_macro_weak_formulation}. Then it holds that
\begin{align}
  \sup_{0\leq t \leq T} |\zeta(t)|_{L^2}^2 &\lesssim |\zeta(t=0) |_{L^2}^2.
  \label{e_estimate_sup_L2} \\
  \int_{0}^\infty |\varphi' (\zeta (t))|_{H^1}^2  dt &\lesssim |\zeta(t=0) |_{L^2}^2.
  \label{e_estimate_integral_varphi'_H1} \\
  \int_{0}^\infty |\zeta (t)|_{H^1}^2  dt &\lesssim |\zeta(t=0) |_{L^2}^2.
  \label{e_estimate_integral_H1}
\end{align}
\end{lemma}
The proof of Lemma~\ref{p_aux_estimates_meso_to_macro} is given in Section~\ref{s_auxiliary_results_section_meso_to_macro_proofs}. The next auxiliary result estimates the difference between the operator $\bar A$ and the second derivative $-\partial_\theta^2$, which allows us to exchange these two operators with some control of error. Intuitively, these two operators are close since $\bar A$ comes from the second difference operator $A$.
\begin{lemma}\label{s_aux_estimates_bar_A_to_partial_2} There exists an integer $K^*$ such that for all $K \geq K^*, M$ and all $y, \tilde{y} \in Y_M$, 
	\begin{align}
	|-\partial_\theta^2 \bar A^{-1} y|_{L^2} &\lesssim |y|_{L^2} , \label{e_norm_bar_A_to_partial_2} \\
	|\langle -\partial_\theta^2 \bar A^{-1} y, \tilde{y} \rangle_{L^2} - \langle y, \tilde{y}\rangle_{L^2} |&\lesssim \frac{1}{K} |y|_{H^{-1}} |\tilde{y}|_{H^1}.
	\label{e_inner_bar_A_to_partial_2}
	\end{align}
\end{lemma}
The proof of Lemma~\ref{s_aux_estimates_bar_A_to_partial_2} is given in Section~\ref{s_spline_approximations}, where we gather and prove facts about splines. Estimate \eqref{e_norm_bar_A_to_partial_2} is closely related to estimate \eqref{opestimate}. The last auxiliary result calculates the time derivative of the projected macroscopic dynamics $P\zeta(t)$.  
\begin{lemma}\label{s_aux_time_derivative_projected_macro} Let~$\zeta(t)$ denote the macroscopic dynamics given by~\eqref{e_macro_weak_formulation}. Then $P\zeta \in H^1_t(Y_M)$, and
	\begin{align}
	\frac{d}{dt} P\zeta &= P\frac{\partial \zeta}{\partial t}.
	\label{e_time_derivative_projected_macro}
	\end{align}
\end{lemma}
The proof of Lemma~\ref{s_aux_time_derivative_projected_macro} is given in Section ~\ref{s_auxiliary_results_section_meso_to_macro_proofs}. We are now ready to prove Theorem ~\ref{p_quantitative_meso_to_macro}.

\begin{proof}[Proof of Theorem~\ref{p_quantitative_meso_to_macro}]
	
	We first bound $\eta - P\zeta$. Because of Lemma~\ref{p_equivalence_neg_H_norm_and_norm_generated_by_neg_bar_A}, we will work with the more natural $\bar{A}^{-1}$ norm instead of $H^{-1}$ norm. By Lemma~\ref{s_aux_time_derivative_projected_macro}, $\eta - P\zeta \in H^1_t(Y_M)$, so product rule gives
	\begin{align*} \frac{d}{dt} \frac{1}{2} |\eta - P\zeta|^2_{\bar{A}^{-1}}  = \langle \frac{d}{dt} \eta  , \eta - P\zeta \rangle_{\bar{A}^{-1}}^2  -  \langle  \frac{d}{dt} P\zeta , \eta - P\zeta \rangle_{\bar{A}^{-1}}.
	\end{align*}
	
	Using the definition of mesoscopic dynamics, the first term becomes
	\begin{equation}
	\langle \frac{d}{dt} \eta  , \eta - P\zeta \rangle_{\bar{A}^{-1}} \overset{\eqref{e_def_mesoscopic_dynamics}}{=} \langle -\bar{A}\nabla \bar H(\eta)  , (\eta - P\zeta) \rangle_{\bar{A}^{-1}} 
	= \langle -\nabla \bar H(\eta), \eta - P\zeta \rangle_{L^2}.
	\end{equation}
	
	Using Lemma~\ref{s_aux_time_derivative_projected_macro} and the definition of the macroscopic dynamics, the second term becomes
	\begin{align}
	-  \langle  \frac{d}{dt} P\zeta , \eta - P\zeta \rangle_{\bar{A}^{-1}} &\overset{\eqref{e_time_derivative_projected_macro}}{=}  -  \langle  P \frac{\partial \zeta}{\partial t}  , \eta - P\zeta \rangle_{\bar{A}^{-1}} \\
	&= -  \langle  \frac{\partial \zeta}{\partial t} , \bar{A}^{-1} (\eta - P\zeta) \rangle_{L^2} \\
	&= -  \langle  \frac{\partial \zeta}{\partial t} ,- \partial_\theta^2 \bar{A}^{-1} (\eta - P\zeta) \rangle_{H^{-1}} \\
	&\overset{\eqref{e_macro_weak_formulation}}{=} \langle  \varphi'(\zeta) ,-\partial_\theta^2 \bar{A}^{-1} (\eta - P\zeta) \rangle_{L^2}.
	\end{align}
	
	Combining these two terms, and then regrouping terms, we get
	\begin{align}\frac{d}{dt} \frac{1}{2} | \eta - P\zeta|^2 _{\bar{A}^{-1}}  =&
	\langle -\nabla \bar H(\eta), \eta - P\zeta \rangle_{L^2} + \langle  \varphi'(\zeta) ,-\partial_\theta^2 \bar{A}^{-1} (\eta - P\zeta) \rangle_{L^2} \\
	 =&  \langle\nabla \bar H(P\zeta) - \nabla \bar H(\eta),  \eta - P\zeta \rangle_{L^2} 
	 \label{e_convexity_meso}\\
	&+ \langle \varphi'(P\zeta)-\nabla \bar H(P\zeta),  \eta - P\zeta \rangle_{L^2}
	\label{e_meso_to_macro_grad}\\
	&+ \langle \varphi'(\zeta) - \varphi'(P\zeta), \eta - P\zeta \rangle_{L^2}
	\label{e_convexity_macro_fluct} \\
	&+ \langle \varphi'(\zeta) - P\varphi'(\zeta), -\partial_\theta^2 \bar{A}^{-1}  (\eta - P\zeta) \rangle_{L^2}
	\label{e_term_1}  \\
	&+ \langle P\varphi'(\zeta), (-\partial_\theta^2 \bar{A}^{-1} - \id ) (\eta - P\zeta) \rangle_{L^2}
	\label{e_term_2} 
	\end{align}
	
	Estimation of the term~\eqref{e_convexity_meso}: by the uniform strict convexity of~$\bar{H} $ (cf.~Theorem~\ref{p_strict_convexity_coarse_grained_Hamiltonian}),
	\begin{align}
	\langle \nabla \bar H ( P \zeta) - \nabla \bar H (\eta) , \eta - P \zeta \rangle_{L^2}   \leq - \lambda | \eta - P \zeta|_{L^2}^2.
	\label{e_convexity_meso_done}
	\end{align}
	
	Estimation of the term~\eqref{e_meso_to_macro_grad}: by convergence of $\nabla \bar{H} (\cdot)$ to $\nabla \mathcal{H} (\cdot)$ (cf. Lemma~\ref{p_convergence_meso_to_macro_free_energy}) and \eqref{e_operator_norm_bounded_H_neg} (cf. Lemma~\ref{p_facts_spline_approximation}),
	\begin{equation}\langle  \varphi' ( P \zeta) - \nabla \bar H ( P \zeta), \eta - P \zeta \rangle_{L^2}
	\lesssim \left(  \left(\frac{1}{K} + \frac{1}{M} \right) |\zeta|_{H^1} + \frac{1}{K}\right) | 
	\eta - P\zeta |_{L^2} 
	\label{e_meso_to_macro_grad_done}
	\end{equation}
	
	Estimation of the term~\eqref{e_convexity_macro_fluct}: by the uniform boundedness of $\varphi''$ (cf. Lemma~\ref{p_strict_convexity_varphi}) and  \eqref{e_approximation_H_neg_1} (cf. Lemma~\ref{p_facts_spline_approximation}),
	\begin{align}
	\langle \varphi'(\zeta) - \varphi'(P\zeta), \eta - P\zeta \rangle_{L^2} \lesssim \frac{1}{M} |\zeta|_{H^1}|\eta - P\zeta|_{L^2}\cdot 
	\label{e_convexity_macro_fluct_done}
	\end{align}
	
	Estimation of the term~\eqref{e_term_1}: by \eqref{e_approximation_H_neg_1} (cf. Lemma~\ref{p_facts_spline_approximation}) and \eqref{e_norm_bar_A_to_partial_2} (cf. Lemma \ref{s_aux_estimates_bar_A_to_partial_2}),
	\begin{align}
	\langle \varphi'(\zeta) - P \varphi'(\zeta), -\partial_\theta^2 \bar{A}^{-1} (\eta - P\zeta) \rangle_{L^2} 
	 \lesssim \frac{1}{M} |\varphi'(\zeta)|_{H^1}|\eta - P\zeta|_{L^2}
	\label{e_term_1_done}
	\end{align}
	
	Estimation of the term~\eqref{e_term_2}: by \eqref{e_inner_bar_A_to_partial_2} (cf. Lemma \ref{s_aux_estimates_bar_A_to_partial_2}), \eqref{e_operator_norm_bounded_H_neg} (cf. Lemma \ref{p_facts_spline_approximation}), Lemma~\ref{p_equivalence_neg_H_norm_and_norm_generated_by_neg_bar_A}, and Poincare inequality,
	\begin{align}
	\langle P\varphi'(\zeta), (-\partial_\theta^2 \bar{A}^{-1} - \id ) (\eta - P\zeta) \rangle_{L^2} 
	\lesssim |\varphi'(\zeta)|_{H^1} \cdot \frac{1}{K} |\eta - P\zeta|_{L^2}
	\label{e_term_2_done}
	\end{align}
	
	Combining the estimates~\eqref{e_convexity_meso_done}, \eqref{e_meso_to_macro_grad_done}, \eqref{e_convexity_macro_fluct_done}
	\eqref{e_term_1_done}, ~\eqref{e_term_2_done} and Young's inequality yields that
\begin{equation}
  \frac{d}{dt} \frac{1}{2} | \eta - P\zeta|^2_{\bar A^{-1}}  \lesssim -\frac{\lambda}{2} |\eta - P\zeta|_{L^2}^2 + \frac{1}{K^2} + \left(\frac{1}{K^2}+\frac{1}{M^2}\right)\left(|\zeta|_{H^1}^2 + |\varphi'(\zeta)|_{H^1}^2\right)
 \end{equation}
 Bringing the term~$-\frac{\lambda}{2} |\eta - P\zeta|_{L^2}^2$ to the left side, integrating in time from $0$ to $T$, applying the energy estimates in Lemma \ref{p_aux_estimates_meso_to_macro}, and exchanging $\bar A^{-1}$ norm with $H^{-1}$ norm (cf. Lemma~\ref{p_equivalence_neg_H_norm_and_norm_generated_by_neg_bar_A}), we get
 \begin{align}
 \sup_{0\leq t\leq T}&\frac{1}{2}|\eta(t) - P\zeta(t)|_{H^{-1}}^2 +  \frac{\lambda}{2} \int_0^T |\eta(t) - P\zeta(t)|_{L^2}^2 dt \\   &\lesssim  \frac{T}{K^2} + \left(\frac{1}{K^2}+\frac{1}{M^2}\right)|\zeta(t=0)|_{L^2}^2
 \label{e_meso_to_macro_final}.
 \end{align}
Since $\eta - \zeta = (\eta - P\zeta) + (P\zeta - \zeta)$, by triangle inequality it remains to bound $P\zeta - \zeta$. This follows from a combination of the spline estimates in Lemma~\ref{p_facts_spline_approximation} and energy estimates in Lemma \ref{p_aux_estimates_meso_to_macro}.

\end{proof}

\section{Auxiliary results}\label{s_micro_to_meso_proof_aux_results}

\subsection{Proof of auxiliary results of Section~\ref{s_central_estimate}} \label{s_aux_results_micro_to_meso}

In this section we give the proof of Lemma~\ref{p_curcialestimate_crucial_formula} and of Lemma~\ref{p_crucial_covariance_estimate}. Before we proceed to the proof, we first describe the gradient~$\nabla_{||}$ on the fibers of~$\ker P =\left\{x \in X_N \ | \ Px =0 \right\} $ (this is deduced in Section~3.3 in ~\cite{DMOW18b}):
\begin{definition}[cf. Definition~3.13 in~\cite{DMOW18b}] \label{d_decomposition_fluctuation_mesoscopic_profile}
 Given~$x \in X_N$, let $x_{\parallel}$ denote the projection of~$x$ onto~$\ker P$ and let~$x_{\perp}$ denote the projection onto~$(\ker P)^{\perp} = \Image NP^t$. They are given by
  \begin{align*}
  x_{\parallel} = x - x_{\perp}  \quad  \mbox{and} \quad    x_{\perp} = NP^t(PNP^t)^{-1} P x.
  \end{align*} 
  \end{definition}
  
\begin{lemma}[cf. Lemma 3.7 in~\cite{DMOW18b}] \label{p_fluctuation_gradient}
  Let~$f: X_N \to \mathbb{R}$ be a smooth function. Let~$\nabla f$ be the gradient inherited from the standard Euclidean structure on~$X_N$. Then the gradient~$\nabla_{\parallel}$ on~$\ker P$ and the gradient ~$\nabla_{\perp}$ on ~$(\ker P)^{\perp} = \Image NP^t$ are given by
  \begin{align}
    \label{e_fluctuation_gradient}
    \nabla_{\parallel} f = \nabla f - \nabla_{\perp} f \quad  \mbox{and} \quad \nabla_{\perp} f = NP^t (PNP^t)^{-1} P \nabla f.
  \end{align}
\end{lemma} 

With this setup, let us now proceed to the argument for Lemma~\ref{p_curcialestimate_crucial_formula}. We need the following auxiliary formula that establishes a link between microscopic and
mesoscopic gradients.

\begin{lemma}\label{nablabar}
We recall that for~$f: X_N\to \mathbb{R}$ and~$y \in Y_M$ we write $\bar f (y) = \int f(x) \mu(dx|y)$ (cf.~\eqref{e_d_bar_f}). Then it holds that
\begin{equation}
\int \nabla f \,\mu (dx|y) \,=\, P^t \nabla \bar{f}(y) \,+\, \operatorname{cov}_{\mu (dx|y)} (f,\nabla H),\label{nablabarf}
\end{equation}
where $\operatorname{cov}_{\mu (dx|y)} (f,\nabla H)$ denotes the
vector with entries $\operatorname{cov}_{\mu(dx|y)} (f,\partial_{x_i} H)$.
\end{lemma}

\begin{proof}[Proof of Lemma~\ref{nablabar}]
  We start by observing that, due to Lemma~\ref{d_decomposition_fluctuation_mesoscopic_profile}, it holds that for any test function~$g:X_N \to \mathbb{R}$
  \begin{align}
    \int g(x) \mu (dx | y) = \frac{\int g(x_{\parallel} + x_{\perp}) \exp (-H(x_{\parallel} + x_{\perp})) dx_{\parallel}}{\int \exp (-H(x_{\parallel} + x_{\perp})) dx_{\parallel}}.
  \end{align}
 Using this formula, we apply integration by parts and get
  \begin{align*}
    \int \nabla_{\parallel} f \mu (dx|y ) &= \int f \nabla_{\parallel} H  \mu (dx|y ), \\
    \int \nabla_{\parallel} H \mu (dx|y ) &= 0,
  \end{align*}
which implies
\begin{equation}\cov_{\mu (dx|y )} \left( f, \nabla_{\parallel} H \right) = \int \nabla_{\parallel} f \mu (dx|y ).
\end{equation}
Thus, we can express the second term of \eqref{nablabarf} as:
\begin{equation}\cov_{\mu (dx|y )} \left( f, \nabla H \right) = \int \nabla_{\parallel} f \mu (dx|y ) + \cov_{\mu (dx|y )} \left( f, \nabla_{\perp} H \right). \label{e_nablabarf_2}
\end{equation}
Write $y = Px$. Since $P^t \nabla \bar f(y) \in \Image P^t = \left(\ker P\right)^{\perp}$, ~$\left( P^t \nabla \bar f(y)\right)_{\parallel} = 0$.
We observe that any $z_{\perp} \in (\ker P)^{\perp}$, 
\begin{equation}
P^t \nabla \bar f (y)\cdot z_\perp =  \langle \nabla \bar f (P x), P z_\perp \rangle
= \nabla (\bar f \circ P) (x) \cdot z_\perp.
\end{equation}
Thus, we can calculate the first term of \eqref{nablabarf} as:
\begin{align}
 P^t \nabla \bar f (y)  &= \nabla (\bar f \circ P) (x)\\
& = \nabla_{\perp} \frac{\int f(x_{\parallel} + x _{\perp}) \exp (-H(x_{\parallel} + x_{\perp} )) dx_{\parallel}}{\int \exp (-H(x_{\parallel} + x_{\perp})) dx_{\parallel}} \\
&= \int   \nabla_{\perp} f \mu(dx |y) -  \int   f \nabla_{\perp} H  \mu(dx |y) \\ 
&\qquad + \int   f  \mu(dx |y)  \int \nabla_{\perp} H \mu(dx |y)  \\
& = \int   \nabla_{\perp} f \mu(dx |y)  - \cov_{\mu(dx|y)} \left(f,  \nabla_{\perp} H \right) \label{e_nablabarf_1}
\end{align}
Combining \eqref{e_nablabarf_1} and \eqref{e_nablabarf_2} gives the desired identity~\eqref{nablabarf}.
\end{proof}

With the help of Lemma~\ref{nablabar} the proof of Lemma~\ref{p_curcialestimate_crucial_formula} consists of a straightforward calculation.
\begin{proof}[Proof of Lemma~\ref{p_curcialestimate_crucial_formula}]
We recall that $P:L^{2}(\mathbb{T}) \to Y_M$ denotes the $L^2$-orthogonal projection onto the subspace~$Y_M$. Direct calculation yields that
  \begin{eqnarray}
   \lefteqn{\frac{d}{dt} \; \int \frac{1}{2} \,\langle P  x-\eta,\bar{A}^{-1} (P  x-\eta)\rangle_{L^2} \,f\, \mu(dx)} \nonumber\\
    &\stackrel{\eqref{micro_evolution}}{=}& -\int \frac{1}{2}\,\nabla \langle P  x-\eta,\bar{A}^{-1} (P  x-\eta)\rangle_{L^2} \cdot A \nabla f \,\mu(dx)\nonumber\\
      && - \int \langle \frac{d\eta}{dt},\bar{A}^{-1} (P x-\eta)\rangle_{L^2} f \,\mu(dx)\nonumber\\
    &\stackrel{\eqref{e_def_mesoscopic_dynamics}}{=}& -\int  P^t \bar{A}^{-1} (P x-\eta) \cdot A \nabla f  \,\mu (dx)\nonumber\\
      && + \int \langle \bar{A} \nabla \bar{H}(\eta),\bar{A}^{-1} (P  x-\eta)\rangle_{L^2} f \,\mu(dx) \nonumber\\
    &=&-\int \langle \bar{A}^{-1} (y-\eta), PA\,\int \nabla f \,\mu (dx|y)\rangle_{L^2} \,\bar{\mu}(dy)\nonumber\\
       && + \int \langle \nabla \bar{H}(\eta),y-\eta\rangle_{L^2} \,\bar{f}\, \bar{\mu}(dy).\nonumber
       \end{eqnarray}
By Lemma ~\ref{nablabar} and integration by parts, the first term becomes       
       \begin{eqnarray}
       \lefteqn{-\int \langle \bar{A}^{-1} (y-\eta), PA\,\int \nabla f \,\mu (dx|y)\rangle_{L^2} \,\bar{\mu}(dy)     } \nonumber \\          
    &\stackrel{\eqref{e_definition_bar_A}}{=}&-\int \langle \bar{A}^{-1} (y-\eta), N^{-1} \bar{A} \nabla \bar{f}\rangle_{L^2} \,\bar{\mu} (dy)         \\
    &&- \int \langle \bar{A}^{-1} (y-\eta),PA\,\operatorname{cov}_{\mu(dx|y)}(f,\nabla H)\rangle_{L^2} \,\bar{\mu} (dy)\nonumber\\
    &=& N^{-1}\int \nabla \cdot \left((y-\eta)\exp(-N\bar{H}(y))\right) \bar{f} \,dy
         \nonumber\\
      &&- \int   AP^t\bar{A}^{-1} (y-\eta) \cdot  \operatorname{cov}_{\mu(dx|y)}(f,\nabla H)  \,\bar{\mu} (dy)\nonumber \\
     &=& \frac{\operatorname{dim}Y_M}{N}
      \,-\, \int \langle y-\eta, \nabla \bar{H}(y) \rangle_{L^2} \bar{f}\, \bar{\mu} (dy) \nonumber \\
      &&- \int   AP^t\bar{A}^{-1} (y-\eta) \cdot  \operatorname{cov}_{\mu(dx|y)}(f,\nabla H)  \,\bar{\mu} (dy).\nonumber
      \end{eqnarray}
Combining the above gives the desired formula~\eqref{e_curcialestimate_crucial_formula}.
\end{proof}

Let us now turn to the verification of Lemma~\ref{p_crucial_covariance_estimate}. This involves a non-trivial ingredient: the conditional measure~$\mu (dx | P x=y)$ satisfies a uniform logarithmic Sobolev inequality (LSI).

\begin{theorem}[Uniform LSI for~$\mu(dx |x=y)$]\label{p_LSI_conditional}
 The conditional measure~$\mu(dx |P x =y)$ given by~\eqref{e_def_Gibbs_measure} satisfies a LSI with constant~$\varrho>0$ uniform in the system size~$N$ and the mesoscopic profile~$y$. More precisely, if~$f: Y_M \to \mathbb{R}$ is a nonnegative test function that satisfies~$\int f (x) \mu (dx |P x=y)=1$ then
 \begin{align}
   \label{e_d_LSI}
 \Ent \left( f \mu( dx | P x=y) | \mu( dx | P x=y) \right) \leq \frac{1}{2\varrho} \int \frac{|\nabla_{||} f(x)|^2}{f(x)} \mu( dx | Px=y),
 \end{align}
where~$|\nabla_{||}f|$ is the norm of the gradient on~$\ker P$ wrt.~the standard Euclidean structure.
\end{theorem}
The logarithmic Sobolev inequality was first discovered by Gross~\cite{Gro75}. It characterizes the speed of convergence to equilibrium of the natural associated drift diffusion process. For more facts about the LSI we refer to the books \cite{R,BGL14} and survey article~\cite{L}. The proof of Theorem~\ref{p_LSI_conditional} is quite subtle. We refer to Theorem 1.8 in the companion article~\cite{DMOW18b}. Using the uniform LSI of Theorem~\ref{p_LSI_conditional} we can derive the following covariance estimate with the help of a standard argument (see Lemma 22 and proof of Proposition 20 in~\cite{GORV}). 
\begin{lemma} \label{covest}
We have:
\begin{equation}
|\operatorname{cov}_{\mu(dx|y)}(f,\nabla H)|^2 \,\leq\,
\frac{\kappa^2}{\rho^2}  \,\bar{f}(y) \, \int \frac{|\nabla_{||} f |^2}{f}  \,\mu(dx|y). \label{covexpl}
\end{equation}
\end{lemma}

Lemma~\ref{covest} almost yields the desired covariance estimate of Lemma~\ref{p_crucial_covariance_estimate}. However, we observe that the right hand side of~\eqref{covexpl} is not the right hand side of~\ref{p_crucial_covariance_estimate}. In order to get the correct right hand side, we have to pass from the Fisher information term involving
$\nabla_{||}$ on the right hand side of \eqref{covexpl} to the full Fisher information term for Kawasaki dynamics. This is done through the following estimate, which is a discrete analogue of the estimate \eqref{e_approximation_H_neg_1}.

\begin{lemma} \label{Poin}
	For $x\in X_N$, let $y \in Y_M$ be the unique solution of $PNP^ty = Px$. There is $\gamma >0$ such that:
	\begin{equation}
	|x_{||}|^2 = \left|x - NP^ty\right|^2  \,\leq\, \frac{\gamma}{M^2}  x \cdot Ax . \label{penalize}
	\end{equation}
\end{lemma}

The proof of Lemma~\ref{Poin} is deferred to Section \ref{s_spline_approximations}, where we gather and prove facts about splines.


\begin{proof}[Proof of Lemma~\ref{p_crucial_covariance_estimate}]
The statement of Lemma~\ref{p_crucial_covariance_estimate} follows now from a combination of Lemma~\ref{covest} and Lemma~\ref{Poin}.
\end{proof}

\subsection{Proofs of auxiliary results of Section~\ref{s_meso_macro}}\label{s_auxiliary_results_section_meso_to_macro_proofs}

In this section we give the proof of  Lemma~\ref{p_aux_estimates_meso_to_macro} and  Lemma~\ref{s_aux_time_derivative_projected_macro}.

\begin{proof}[Proof of Lemma~\ref{p_aux_estimates_meso_to_macro}]
By Lemma \ref{p_strict_convexity_varphi}	, $\varphi$ is convex, differentiable, and $\varphi'(0) = 0$, so $\varphi$ takes minimum at $\theta = 0$. Thus, $\mathcal{H}$ takes minimum at $\zeta \equiv 0$ and $\min \mathcal{H} = \varphi(0)$. Moreover, the uniform bounds for $\varphi''$ implies
\begin{align}\lambda |\zeta|_{L^2}^2 \leq \mathcal{H} (\zeta) - \min \mathcal{H} \leq \Lambda |\zeta|_{L^2}^2.
\label{e_bounds_for_macro}
\end{align}

Using the definition of macroscopic dynamics, we find
\begin{align}\label{e_hydro_equ_numeric_lemma_aux}
\frac{d}{dt} \mathcal{H} (\zeta) &= \frac{d}{dt} \int \varphi (\zeta(t, \theta)) d \theta \\
&= \int \varphi' (\zeta) \frac{\partial \zeta}{\partial t} d \theta \\
&\overset{\eqref{hydro_equn}}{=} \int \varphi' (\zeta) \partial_\theta^2 \left( \varphi' (\zeta) \right) d \theta \\
&= - |\varphi' (\zeta(t))|_{H^1}^2  \leq 0.
\label{e_time_derivative_macro}
\end{align}
Thus, $\mathcal{H}(\zeta(t))$ is decreasing in $t$, so by \eqref{e_bounds_for_macro}
\begin{equation}
 |\zeta(t)|_{L^2}^2 \lesssim \mathcal{H}(\zeta(t)) - \min \mathcal{H} 
  \leq \mathcal{H}(\zeta(0)) - \min \mathcal{H} 
   \lesssim |\zeta(0)|_{L^2}^2 
\end{equation}
which implies \eqref{e_estimate_sup_L2}. For \eqref{e_estimate_integral_varphi'_H1}, using~\eqref{e_time_derivative_macro} it follows that
\begin{equation}
  \int_0^\infty |\varphi' \left( \zeta(t) \right)|_{H^1}^2 dt  
  = \lim_{T\rightarrow \infty} \mathcal{H}(\zeta(0)) - \mathcal{H}(\zeta(T)) 
  \leq \mathcal{H}(\zeta(0)) - \min \mathcal{H} 
  \lesssim |\zeta(0)|_{L^2}^2 , 
\end{equation}
as to be shown. Finally, for \eqref{e_estimate_integral_H1}, observe that
\begin{align}
|\varphi'(\zeta)|_{H^1}^2 = |\partial_\theta(\varphi'(\zeta))|_{L^2}^2 
= |\varphi''(\zeta) \partial_\theta \zeta|_{L^2}^2 
\geq \lambda^2 |\partial_\theta \zeta|_{L^2}^2 
= \lambda^2 |\zeta|_{H^1}^2,
\end{align}
which gives
\begin{align}
\int_0^\infty |\zeta|_{H^1}^2 dt \lesssim  \int_0^\infty |\varphi' \left( \zeta(t) \right)|_{H^1}^2 dt \lesssim |\zeta (0)|_{L^2}^2
\end{align}
as to be shown.

\end{proof}

\begin{proof}[Proof of Lemma~\ref{s_aux_time_derivative_projected_macro}]

Since $\zeta \in L^\infty_t(L^2_\theta)$ and $\|P\|_{L^2_\theta \rightarrow L^2_\theta} = 1$, $P\zeta \in L^\infty_t(Y_M)$. Since $\displaystyle\frac{\partial \zeta}{\partial t} \in L^2_t(H^{-1}_\theta)$ and $\|P\|_{H^{-1}_\theta \rightarrow H^{-1}_\theta} < \infty$, $P\zeta \in L^\infty_t(Y_M)$ because norms on a finite dimensional space are equivalent. Thus, it remains to verify \eqref{e_time_derivative_projected_macro}. First, it's easy to check that for any $\xi \in C^1([0, T], Y_M)$, \begin{align}\label{e_aux_derivative_in_two_senses}
\frac{\partial \xi}{\partial t}(t, \theta) = \frac{d \xi}{d t}(t) (\theta).
\end{align}
Using this identity, it's straightforward to verify \eqref{e_time_derivative_projected_macro} by checking the definition of weak derivatives. 
\end{proof}

\subsection{Properties of spline approximations}\label{s_spline_approximations}
In this section we gather and prove the facts about splines~$y \in Y_M$ needed in this article. More precisely, we prove Lemma~\ref{p_equivalence_neg_H_norm_and_norm_generated_by_neg_bar_A}, Lemma~\ref{operator}, Lemma~\ref{s_aux_estimates_bar_A_to_partial_2}, and Lemma~\ref{Poin}. 

For the proof of Lemma~\ref{p_equivalence_neg_H_norm_and_norm_generated_by_neg_bar_A}, we need two auxiliary results. The first is an inverse Sobolev inequality on the space~$Y_M$ (cf. Lemma 4.8 in \cite{DMOW18b}).
\begin{lemma}[Inverse Sobolev inequality] \label{p_inverse_sobolev}
	For all~$y \in Y_M$ holds
	\begin{align}\label{e_inverse_sobolev}
	|y|_{H^{2}} \lesssim M  |y|_{H^{1}} \lesssim M^2 |y|_{L^2}.
	\end{align}
\end{lemma}

The second auxiliary result we need is that the~$H^{1}$ inner product is close to the inner product induced by the positive definite operator ~$\bar{A}$
\[
\langle y, z \rangle_{\bar{A}} = \langle y , \bar{A} z \rangle_{L^2} \quad \mbox{and} \quad |y|_{\bar{A}} = \sqrt{\langle y , \bar{A} y \rangle_{L^2}}.
\]
\begin{lemma}\label{p_closeness_H_1_and_A_norm_concrete}
	There exists an integer $K^*$ such that for all $K\geq K^*, M$ and all~$y, \tilde y \in Y_M$
	\begin{align}
	| \langle \tilde y, \bar A  y \rangle_{L^2} - \langle \tilde y , y \rangle_{H^1} | & \lesssim  \frac{1}{N} \left( | \tilde y|_{H^{1}} |y|_{H^{2}}  + | \tilde y|_{H^{2}} |y|_{H^{1}}  \right)  \label{e_closeness_H_1_and_A_norm_concrete_weak} \\
	&\lesssim \frac{M}{N}| \tilde y|_{H^{1}} |y|_{H^{1}}. \label{e_closeness_H_1_and_A_norm_concrete}
	\end{align}
\end{lemma}

\begin{corollary}\label{p_closeness_H_1_and_A_norm_concrete_cor} There exists an integer $K^*$ such that for all $K\geq K^*, M$ and $y \in Y_M$,
	\begin{align}
	|y|_{H^1} \simeq |y|_{\bar A}.
	\label{e_equivalence_H_norm_and_bar_A}
	\end{align}
\end{corollary}
This leads to a quick proof of Lemma~\ref{p_equivalence_neg_H_norm_and_norm_generated_by_neg_bar_A} by a duality argument.

\begin{proof}[Proof of Lemma~\ref{p_equivalence_neg_H_norm_and_norm_generated_by_neg_bar_A}] Let $z \in H^1(\mathbb{T})$ be arbitrary, then we have
	\begin{align}
	\langle y, z\rangle_{L^2} = \langle y, Pz\rangle_{L^2} 
	\leq |y|_{\bar A^{-1}} |Pz|_{\bar A} 
	\overset{\eqref{e_equivalence_H_norm_and_bar_A}}{\lesssim} |y|_{\bar A^{-1}} |Pz|_{H^1} 
	\overset{\eqref{e_operator_norm_bounded_H_neg}}{\lesssim} |y|_{\bar A^{-1}} |z|_{H^1}
	\end{align}
	which implies $|y|_{H^{-1}} \lesssim |y|_{\bar A^{-1}}$. To show the opposite inequality, let $w \in Y_M$ be arbitrary, then we have
	\begin{align}
	\langle y, \bar A^{-1} w\rangle_{L^2} \leq |y|_{H^{-1}} |\bar A^{-1} w|_{H^1} 
	\overset{\eqref{e_equivalence_H_norm_and_bar_A}}{\lesssim} |y|_{H^{-1}} |\bar A^{-1} w|_{\bar A} 
	\lesssim |y|_{H^{-1}} |w|_{\bar A^{-1}}
	\end{align}
	which implies $|y|_{\bar A^{-1}} \lesssim |y|_{H ^{-1}}$. 
\end{proof}

Let us now prove Lemma~\ref{p_closeness_H_1_and_A_norm_concrete}. The main ingredient of the argument is the inverse Sobolev inequality for~$Y_M$ in Lemma~\ref{p_inverse_sobolev}. 
\begin{proof}[Proof of Lemma~\ref{p_closeness_H_1_and_A_norm_concrete}] 
	We will prove the estimate~\eqref{e_closeness_H_1_and_A_norm_concrete_weak}. The estimate~\eqref{e_closeness_H_1_and_A_norm_concrete} follows from~\eqref{e_closeness_H_1_and_A_norm_concrete_weak} by an application of~\eqref{e_inverse_sobolev}. \smallskip
	
	Let us recall that the operator~$\bar A: Y_M \to Y_M $ is given by~$\bar A  = P ANP^t$, where~$A$ is the second order difference operator given by~\eqref{e_def_A} and~$P, NP^t$ are an adjoint pair of $L^2$-orthogonal projections defined in Definitions \ref{d_coarse_graining_operator_P} and \ref{d_adjoint_NPt}. It follows from the definition of $A$  that for two vectors $\tilde x, x \in \mathbb{R}^N$ 
	\begin{align*}
	\frac{1}{N} \tilde x \cdot A x \overset{\eqref{e_def_A}}{=} N \sum_{i=1}^N \left( \tilde x_i  - \tilde x_{i-1}\right) \left( x_i - x_{i-1}\right).
	\end{align*}
	It follows from the explicit formula for $NP^t$ that for $i \in \{1,2 ,\cdots, N\}$
	\begin{align*}
	 (NP^t y)_i - (NP^t y)_{i-1}    \overset{\eqref{e_adjoint_NP^t_explicit}}{=}  \int_{\frac{i-2}{N}}^{\frac{i-1}{N}}  \partial_{\theta}^{\frac{1}{N}}y(\theta) d \theta ,
	\end{align*}
	where~$\partial_{\theta}^{\frac{1}{N}}$ is the forward difference quotient
	\begin{align*}
	\partial_{\theta}^{\frac{1}{N}}y(\theta) = N \left( y \left( \theta + \frac{1}{N} \right) - y \left( \theta\right) \right).
	\end{align*}
	Hence, we have
	\begin{align} \label{e_rep_a_inner_product}
	\langle \tilde y , \bar A y \rangle_{L^2} &= \frac{1}{N}   \left( N P^t \tilde y \right) \cdot A \left( N P^t y \right) = \int_0^1 \partial_{\theta}^{\frac{1}{N}} \tilde y \ Q \partial_{\theta}^{\frac{1}{N}}y \ d \theta,
	\end{align}
	where~$Q : L^2(\mathbb{T}) \to \bar X_N$ is the~$L^2$-orthogonal projection onto the piecewise constant functions
	\begin{align*}
	\bar X_N= \left\{ f: \mathbb{T} \to \mathbb{R} \ : \  f \mbox{ is constant on } \left(\frac{i-1}{N}, \frac{i}{N} \right) , \ i=1 \ldots , N  \right\}.
	\end{align*}
	Using~\eqref{e_rep_a_inner_product} we get that
	\begin{align*}
	 \langle \tilde y,  \bar A y \rangle_{L^2} -  \langle \tilde y , y \rangle_{H^1} &= \int_0^1  \partial_{\theta}^{\frac{1}{N}} \tilde y Q \partial_{\theta}^{\frac{1}{N}} y  d \theta - \int_0^1 \partial_{\theta} \tilde y  \ \partial_{\theta} y  d \theta \\
	& = \int_0^1 \left( \partial_{\theta}^{\frac{1}{N}} \tilde y - \partial_{\theta} \tilde y \right) Q \partial_{\theta}^{\frac{1}{N}} y d \theta \\
	& \quad + \int_{0}^1 \partial_{\theta} \tilde y Q \left( \partial_{\theta}^{\frac{1}{N}} y - \partial_{\theta} y \right) d \theta \\
	& \quad + \int_0^1 \partial_{\theta} \tilde y \left(Q - \Id  \right) \partial_{\theta} y d \theta.
	\end{align*}
	
	The last identity yields the estimate
	\begin{align}
	\left| \langle \tilde y,  \bar A y \rangle_{L^2} -  \langle \tilde y , y \rangle_{H^1} \right| &  \leq \left( \int_0^1 \left| \partial_{\theta}^{\frac{1}{N}} \tilde y - \partial_{\theta} \tilde y \right|^2 d \theta \int_0^1 \left| \partial_{\theta}^{\frac{1}{N}} y  \right|^2 d \theta \right)^{\frac{1}{2}} \\
	& \qquad + \left( \int_0^1 \left| \partial_{\theta} \tilde y \right|^2 d \theta \int_0^1 \left| \partial_{\theta}^{\frac{1}{N}} y - \partial_{\theta} y \right|^2 d \theta \right)^{\frac{1}{2}} \\
	& \qquad + \left( \int_0^1 \left| \partial_{\theta} \tilde y \right|^2 d \theta \int_0^1 \left| \left( Q - \Id \right) \partial_{\theta} y \right|^2 d \theta \right)^{\frac{1}{2}} . \label{e_comparison_bar_a_H_1_norm_core_estimate}
	\end{align}
	We will estimate the integrals on the right hand side. Observe that
	\begin{align}
	\partial_{\theta}^{\frac{1}{N}} y(\theta)  = N \int_\theta^{\theta + \frac{1}{N}} \partial_\theta y(s) ds 
	 = \partial_{\theta} y \left( \tilde s\right) \mbox{ for } \tilde s \in \left(\theta, \theta+ \frac{1}{N} \right). \label{e_partial_N_integral_formula}
	\end{align}
	Using the first identity in \eqref{e_partial_N_integral_formula} and Hoelder's inequality, we have
	\begin{align}
	\left| \partial_{\theta}^{\frac{1}{N}} y(\theta) \right|^2 & \leq  N  \int_{\theta}^{\theta + \frac{1}{N}} \left| \partial_{\theta} y \left(s\right)\right|^2 ds.
	\end{align}
	Using periodicity, this implies
	\begin{align}\label{e_A_inner_product_first_estimate}
	\int_0^1 \left| \partial_{\theta}^{\frac{1}{N}} y \right|^2 d \theta   \leq  N \int_{0}^1 \int_{\theta}^{\theta + \frac{1}{N}} \left| \partial_{\theta} y \left(s\right)\right|^2 ds d \theta  = \int_0^1 \left| \partial_{\theta} y \left(s\right)\right|^2 ds.
	\end{align}
	Next, using the second identity in ~\eqref{e_partial_N_integral_formula} we have
	\begin{align}
	\left| \partial_{\theta}^{\frac{1}{N}} y - \partial_{\theta} y \right|  = \left| \partial_{\theta}y (\tilde s)- \partial_{\theta} y(\theta) \right|  
	=   \left| \int_{\theta}^{\tilde s} \partial_{\theta}^2 y (s) ds \right| \leq    \int_{\theta}^{\theta + \frac{1}{N}} |\partial_{\theta}^2 y (s)| ds.
	\end{align}
	Integrating this inequality yields
	\begin{align}
	 \int_0^1 \left| \partial_{\theta}^{\frac{1}{N}} y - \partial_{\theta} y \right|^2 d \theta 
	& \leq \int_0^1  \left( \int_{\theta}^{\theta + \frac{1}{N}} \left|\partial_{\theta}^2 y (s) \right| ds \right)^2 d \theta \\
	&\leq \frac{1}{N} \int_0^1   \int_{\theta}^{\theta + \frac{1}{N}} \left|\partial_{\theta}^2 y (s) \right|^2 ds  d \theta  \\
	&= \frac{1}{N^2} \int_0^1   \left|\partial_{\theta}^2 y (s) \right|^2 ds. \label{e_A_inner_product_second_estimate}
	\end{align} 	
	Finally, Poincar\'e inequality on the interval $(\frac{i-1}{N}, \frac{i}{N})$ applied to $\partial_\theta y$ gives
	\begin{align}
	\label{e_A_inner_product_third_estimate}
	\int_0^1 \left| \left( Q - \Id \right) \partial_{\theta} y\right|^2 d \theta \lesssim  \frac{1}{N^2} \int_0^1 \left|  \partial_{\theta}^2 y\right|^2 d \theta.
	\end{align}
	Applying the estimates~\eqref{e_A_inner_product_first_estimate},~\eqref{e_A_inner_product_second_estimate} and~\eqref{e_A_inner_product_third_estimate} to the right hand side of~\eqref{e_comparison_bar_a_H_1_norm_core_estimate} yields the desired estimate~\eqref{e_closeness_H_1_and_A_norm_concrete_weak}.
\end{proof}

\begin{proof}[Proof of Lemma~\ref{operator}]
	We will work with the B-spline basis of $Y_M$. It is defined as
	the family of quadratic spline functions on $\mathbb{T}^1$ given by 
	\begin{equation} \label{bspline}
	B_j(\theta) =
	\begin{cases}
	\frac{M^2}{2} (\theta-\frac {j-2}{M})^2 & \text{for}\, \theta\in [\frac{j-2}{M},\frac{j-1}{M}) \\
	\frac{3}{4} \,-\,M^2(\theta - \frac{2j-1}{2M})^2 & \text{for}\, \theta\in [\frac{j-1}{M},\frac{j}{M}) \\
	\frac{M^2}{2} (\theta-\frac {j+1}{M})^2 & \text{for}\, \theta \in [\frac{j}{M},\frac{j+1}{M}) \\
	0 & \text{else}.
	\end{cases}
	\end{equation} 
	We will deduce the following estimate which directly yields~\eqref{opestimate}.
	\begin{equation}\label{e_opestimate_2}
	\sigma \, |ANP^ty |_{L^2} \leq |P ANP^ty |_{L^2}. 
	\end{equation}
	
	We begin by giving an explicit formula of $z = ANP^t y$ in terms of $y$. By definition, in each of the intervals $[\frac{j-1}{M},\frac{j}{M})$, $y\in Y_M$ is of the form $$y(\theta) =\alpha_j \theta^2 + \beta_j \theta +\gamma_j$$
	for some coefficients $\alpha_j,\beta_j,\gamma_j \in \mathbb{R}$. Then for all $j=1,...,M$ and
	$i=(j-1)K+1,...,jK$
	\begin{align}
	(NP^{t}y)_i   \overset{\eqref{e_adjoint_NP^t_explicit}}{=}  N\int_{\frac{i-1}{N}}^{\frac{i}{N}} \,y\,d\theta 
	 =\frac{\alpha_j}{N^2} \left(i^2-i+\frac{1}{3}\right) \,+\, \frac{\beta_j}{N} \left(i-\frac{1}{2}\right) \,+\,\gamma_j.
	\end{align}
	Consequently, recalling the definition of $A$ in \eqref{e_def_A}, we find for all entries of $z = ANP^ty$ away from
	the block boundaries, that is $i=(j-1)K+2,...,jK-1$:
	\begin{equation}
	z_i = (ANP^ty)_i =-2\alpha_j. \label{zinner}
	\end{equation}
	The fact that $y$ belongs to $C^1(\mathbb{T})$ yields the conditions:
	\begin{equation}
	N(\beta_{j}-\beta_{j+1})\,=\,2 jK (\alpha_{j+1}-\alpha_j), \label{C1}
	\end{equation}
	\begin{equation*}
	N^2(\gamma_j-\gamma_{j+1}) \,=\, N(\beta_{j+1}-\beta_j) jK \,+\, (\alpha_{j+1}-\alpha_j)(jK)^2
	\,\stackrel{\eqref{C1}}{=}\, - (\alpha_{j+1}-\alpha_j)(jK)^2. \label{C0}
	\end{equation*}
	A straightforward computation then shows for the values of $z$ at the boundaries of the blocks:
	\begin{align}
	z_{jK} &= -2\alpha_j + \frac{1}{3}(\alpha_j-\alpha_{j+1})
	\label{boundary1} \\
	z_{(j-1)K+1} &= -2\alpha_j + \frac{1}{3}(\alpha_j-\alpha_{j-1}). 
	\label{boundary2}
	\end{align}
	So, the function $z$ is almost piecewise constant on the mesh $\{\frac{m}{M}\}_{m=1,...,M}$ and we define a spline interpolation
	for $z$ in the simplest way one can imagine via
	\begin{equation*}
	I(z) := \sum_{j=1}^{M} -2\alpha_j\,B_j.
	\end{equation*}
	We will show that there is a universal constant $\sigma>0$, such that:
	\begin{equation}
	\frac{\langle z,I(z)\rangle_{L^2}}{|z|_{L^2} |I(z)|_{L^2}} \geq \sigma, \label{opest}
	\end{equation}
	which implies our claim \eqref{e_opestimate_2} by the following simple calculation:
	\begin{equation*}
	|Pz|_{L^2}\,|I(z)|_{L^2} \geq \langle Pz,I(z)\rangle_{L^2} = \langle z,I(z) \rangle_{L^2} \geq \sigma |z|_{L^2}|I(z)|_{L^2}.
	\end{equation*}
	\newline
	Argument for \eqref{opest}: From \eqref{zinner}, \eqref{boundary1}
	and \eqref{boundary2}, it follows that there is $C<\infty$ (depending on $K$ and going to $4$ as $K\rightarrow\infty$), such that:
	\begin{eqnarray*}
		|z|^2_{L^2} = \frac{1}{N} |z|^2  \leq C\frac{K}{N}\sum_{j=1}^M \alpha_j^2.
	\end{eqnarray*}
	Next, we note that
	\begin{equation}
	\langle B_j,B_k \rangle_{Y} =
	\begin{cases}
	\frac{1}{M} \frac{11}{20} & \text{for} \,j=k \\
	\frac{1}{M} \frac{13}{60} & \text{for} \,|j-k|=1 \\
	\frac{1}{M} \frac{1}{120} & \text{for} \,|j-k|=2 \\
	0 & \text{else}. \label{BjBk}
	\end{cases}
	\end{equation}
	This implies:
	\begin{eqnarray*}
		|I(z)|^2_Y = \langle \sum_{j=1}^{M} -2\alpha_j B_j,\sum_{k=1}^M -2\alpha_k B_k \rangle_Y = \frac{4}{M}\langle \alpha,D \alpha\rangle_{\mathbb{R}^M},
	\end{eqnarray*}
	where $D$ is a symmetric matrix with $\|D\|_2 \leq 1$. Hence,
	\begin{equation*}
	|I(z)|^2_{L^2} \leq \frac{4}{M}\sum_{j=1}^M \alpha_j^2.
	\end{equation*}
	Finally, we compute $\langle z,I(z) \rangle_{L^2}$:
	\begin{eqnarray}
	\lefteqn{\langle z,I(z) \rangle_{L^2} = \langle z,\sum_{j=1}^{M} -2\alpha_j B_j\rangle_{L^2}
		= \sum_{j=1}^{M} -2\alpha_j \,\langle z,B_j\rangle_{L^2}}\nonumber\\
	&=& \sum_{j=1}^{M} -2\alpha_j \left(-2\alpha_{j-1} \int_{\frac{j-2}{M}}^{\frac{j-1}{M}} B_j \,d\theta
	\,-\, 2\alpha_j \int_{\frac{j-1}{M}}^{\frac{j}{M}} B_j \,d\theta \,-\, 2\alpha_{j+1} \int_{\frac{j}{M}}^{\frac{j+1}{M}} B_j \,d\theta\right)\nonumber\\
	&& +\, \sum_{j=1}^{M} -2\alpha_j \,\left(\sum_{k=j-1}^{j+1}\frac{1}{3} (\alpha_k-\alpha_{k-1})
	\int_{\frac{(k-1)K}{N}}^{\frac{(k-1)K+1}{N}} B_j \,d\theta\right)\nonumber\\
	&& +\, \sum_{j=1}^{M} -2\alpha_j \,\left(\sum_{k=j-1}^{j+1}\frac{1}{3} (\alpha_k-\alpha_{k+1})
	\int_{\frac{kK-1}{N}}^{\frac{kK}{N}} B_j \,d\theta\right)\nonumber\\
	&=& \sum_{j=1}^{M} \left(\frac{2}{3M} \alpha_{j-1}\alpha_j + \frac{8}{3M} \alpha_j^2 + \frac{2}{3M} \alpha_{j+1}\alpha_j \right)
	+ O\left(\frac{1}{N}\right) |\alpha |_{\mathbb{R}^M}^2.
	\end{eqnarray}
	The strict diagonal dominance of the symmetric matrix $E_{jk}= \frac{2}{3} \delta_{j-1,k} + \frac{8}{3} \delta_{j,k}+ \frac{2}{3}\delta_{j+1,k}$ thus
	implies that there is $c > 0$ (dependent on $K$ and going to $\frac{4}{3}$ as $K \rightarrow \infty$), such that:
	\begin{equation*}
	\langle z,I(z) \rangle_{L^2(\mathbb{T})} \geq \frac{c}{M} \sum_{j=1}^M \alpha_j^2.
	\end{equation*}
	Putting everything together, we arrive at \eqref{opest}:
	\begin{equation*}
	\frac{\langle z,I(z)\rangle_{L^2}^2}{|z |_{L^2}^2\, |I(z)|_{L^2}^2} \geq \frac{c^2}{M^2} \frac{M}{4}\frac{N}{CK} = \frac{c^2}{4C}.
	\end{equation*}
	
\end{proof}

\begin{proof}[Proof of Lemma~\ref{s_aux_estimates_bar_A_to_partial_2}]
	
	Argument for ~\eqref{e_norm_bar_A_to_partial_2}: We will show the following estimate which directly yields ~\eqref{e_norm_bar_A_to_partial_2} in Lemma~\ref{s_aux_estimates_bar_A_to_partial_2}:
	\begin{equation}\label{e_norm_bar_A_to_partial_2_alternate}
	\, |-\partial_\theta^2 y |_{L^2} \lesssim |P ANP^ty |_{L^2}. 
	\end{equation}
	This will be an easy consequence of the estimate \eqref{e_opestimate_2} and its proof above. Observe that the function $-\partial_\theta^2 y$ is piecewise constant, taking the value $-2\alpha_j$ on the interval $[\frac{j-1}{M}, \frac{j}{M})$. Compare this with the explicit description of the function $ANP^t y$ given by \eqref{zinner}, \eqref{boundary1}, \eqref{boundary2}, we find
	\begin{equation}
	(-\partial_\theta^2 y)(\theta) - ANP^t y(\theta) = 
	\begin{cases}
	\frac{1}{3}(\alpha_{j+1} - \alpha_j) & \mbox{for} \,  \theta \in [\frac{jK-1}{N}, \frac{jK}{N}) \\
	\frac{1}{3}(\alpha_{j-1} - \alpha_j) & \mbox{for} \, \theta \in [\frac{(j-1)K}{N},  \frac{(j-1)K+1}{N}) \\
	0 & \mbox{otherwise}
	\end{cases}
	\end{equation}
	Then it's easy to see that
	\begin{align}
	|(-\partial_\theta^2 y) - ANP^t y|_{L^2}^2 &\leq \sum_{j=1}^M \frac{1}{N} |\alpha_j|^2
	= \frac{1}{K} \frac{1}{M} \sum_{j=1}^M |\alpha_j|^2  
	\leq \frac{1}{4K} |-\partial_\theta^2 y|_{L^2}^2.
	\end{align}
	It follows that
	\begin{align}
	|-\partial_\theta^2 y|_{L^2} \leq \left(1+O\left(\frac{1}{K^{\frac{1}{2}}}\right)\right) |ANP^t y|_{L^2}, 
	\end{align}
	which, combined with the estimate \eqref{e_opestimate_2}, implies \eqref{e_norm_bar_A_to_partial_2_alternate}. \\
	
	Argument for ~\eqref{e_inner_bar_A_to_partial_2}: we apply Lemma~\ref{p_closeness_H_1_and_A_norm_concrete} to $\bar A^{-1} y, \tilde{y}$:
	\begin{align}
	|\langle \tilde{y},  y\rangle_{L^2} - \langle \tilde{y}, -\partial_\theta^2 \bar A^{-1} y \rangle_{L^2}| 
	&= |\langle \tilde{y}, \bar A (\bar A^{-1} y) \rangle_{L^2} - \langle \tilde{y}, \bar A^{-1} y \rangle_{H^1}| \\
	&\overset{\eqref{e_closeness_H_1_and_A_norm_concrete_weak}}{\lesssim} \frac{1}{K} |\bar A^{-1} y|_{H^1} |\tilde{y}|_{H^1}
	\end{align}
This is almost ~\eqref{e_inner_bar_A_to_partial_2}, and the proof is completed by equivalence of norms:
\begin{equation}|\bar A^{-1} y|_{H^1} \overset{\eqref{e_equivalence_H_norm_and_bar_A}}{\lesssim} |\bar A^{-1} y|_{\bar A} 
=  |y|_{\bar A^{-1}}
\overset{\eqref{e_equivalence_neg_H_norm_and_neg_bar_A}}{\lesssim} |y|_{H^{-1}}. 
\end{equation}
\end{proof}

For the proof of Lemma~\ref{Poin} we need one auxiliary statement. It is a well-known discrete analogue of the Poincar\'e inequality for functions with mean zero.
\begin{lemma}[Discrete Poincar\'e inequality]\label{p_discrete_poincare}
	There is a universal constant $0<C<\infty$ such that for all $x\in \mathbb{R}^N$ with $\sum_{n=1}^N x_n =0$
	it holds
	\begin{equation}
	\sum_{n=1}^N x_n^2 \leq CN^2 \sum_{n=2}^{N} (x_n-x_{n-1})^2. \label{discretepoincare}
	\end{equation}
\end{lemma}

\begin{proof}[Proof of Lemma~\ref{Poin}]
	 Let us consider the following spline interpolation of an element ~$x \in X_N$:
	\begin{equation*}
	I(x) := NP^t \left[\sum_{j=1}^M \left(\frac{1}{3K} \sum_{i=(j-2)K+1}^{(j+1)K} x_i\right) B_j\right],
	\end{equation*}
	where~$B_j$ is given by~\eqref{bspline}. Since $NP^t$ is $L^2$-orthogonal projection onto $X_N$, it holds
	\begin{equation*}
	\sum_{n=1}^N (x_n -(NP^ty)_n)^2 \,\leq\, \sum_{n=1}^N (x_n -I(x)_n)^2.
	\end{equation*}
	Using~$\sum_{j=1}^M B_j = 1$ and Young's inequality we compute
	\begin{eqnarray}
	\lefteqn{\sum_{n=(m-1)K+1}^{mK} (x_n - I(x)_n)^2}\nonumber\\
	&\leq& 3 \sum_{n=(m-1)K+1}^{mK} \left(x_n - \left(\frac{1}{3K} \sum_{i=(m-3)K+1}^{mK} x_i\right)\right)^2 (NP^tB_{m-1})_n^2\nonumber\\
	&&+ \;\; 3 \sum_{n=(m-1)K+1}^{mK} \left(x_n - \left(\frac{1}{3K} \sum_{i=(m-2)K+1}^{(m+1)K} x_i\right)\right)^2  (NP^tB_{m})_n^2\nonumber\\
	&&+ \;\; 3 \sum_{n=(m-1)K+1}^{mK} \left(x_n - \left(\frac{1}{3K} \sum_{i=(m-1)K+1}^{(m+2)K} x_i\right)\right)^2  (NP^tB_{m+1})_n^2\nonumber\\
	&\leq& 3 \sum_{n=(m-3)K+1}^{mK} \left(x_n - \left(\frac{1}{3K} \sum_{i=(m-3)K+1}^{mK} x_i\right)\right)^2 \nonumber\\
	&&+ \;\; 3 \sum_{n=(m-2)K+1}^{(m+1)K} \left(x_n - \left(\frac{1}{3K} \sum_{i=(m-2)K+1}^{(m+1)K} x_i\right)\right)^2  \nonumber\\
	&&+ \;\; 3 \sum_{n=(m-1)K+1}^{(m+2)K} \left(x_n - \left(\frac{1}{3K} \sum_{i=(m-1)K+1}^{(m+2)K} x_i\right)\right)^2 \nonumber\\
	&\stackrel{\eqref{discretepoincare}}{\leq}& 3C(3K)^2 \left(\sum_{n=(m-3)K+2}^{mK} (x_n - x_{n-1})^2\right)\nonumber\\
	&&+ \;\; 3C(3K)^2 \left(\sum_{n=(m-2)K+2}^{(m+1)K} (x_n - x_{n-1})^2\right)\nonumber\\
	&&+ \;\; 3C(3K)^2 \left(\sum_{n=(m-1)K+2}^{(m+2)K} (x_n - x_{n-1})^2\right). \label{poinblock}
	\end{eqnarray}
	Observing that
	\begin{equation*}
	\langle x,Ax\rangle \,=\, N^2 \sum_{n=1}^N (x_n-x_{n-1})^2,
	\end{equation*}
	we add the inequalities \eqref{poinblock} for all $m\in [M]$ and arrive at \eqref{penalize} with $\gamma=3^5\,C$,
	where $C$ is the constant from \eqref{discretepoincare}.
\end{proof}

\section*{Acknowledgment}
This research has been partially supported by NSF grant DMS-1407558. Georg Menz and Tianqi Wu want to thank the Max-Planck Institute for Mathematics in the Sciences, Leipzig, Germany, for financial support.

\bibliographystyle{alpha}

\bibliography{bib}

\end{document}